\documentclass[11pt]{amsart}

\usepackage{amsmath,amssymb,amsfonts,amsthm,enumerate}
\usepackage{vmargin}
\usepackage[dvips]{graphicx}
\usepackage{color}
\input{amssym.def}
\input{amssym}
\input{epsf}
\usepackage{a4wide}
\usepackage{supertabular}
\usepackage{array}
\usepackage{multicol}

\usepackage[titletoc]{appendix}

\setmarginsrb{3cm}{2cm}{3cm}{3cm}{75pt}{20pt}{20pt}{30mm}
\def\R{\mathbb R}
\def\N{\mathbb N}

\def\Sph{\mathbb S}
\def\ov{\overline}
\def\cal{\mathcal}

\def\hat{\widehat}

\def\tilde{\widetilde}

\def\TM{T\!M}
\def\UM{U\!M}

\def\TFL{{\rm TFL}}

\def\dist{{\rm dist}}

\def\I{{\rm I}}
\def\NF{{\rm NF}}

\def\J{\mathbb J}
\def\Ker{{\rm Ker}} 
\def\vect{{\rm Vect}}

\renewcommand{\bar}{\overline}

\def\S{\mathfrak{S}}
\def\<{\langle}
\def\>{\rangle}

\def\TCL{{\rm TCL}}

\def\cut{{\rm cut}}

\def\TCP{{$\mathcal{TCP}$ }}
\def\PP{\mathcal P}

\def\begeq{\begin{equation}}
\def\endeq{\end{equation}}
\def\begar{\begin{eqnarray}}
\def\endar{\end{eqnarray}}
\def\begar*{\begin{eqnarray*}}
\def\endar*{\end{eqnarray*}}
\def\begal{\begin{align}}
\def\endal{\end{align}}
\def\begal*{\begin{align*}}
\def\endal*{\end{align*}}

\newtheorem{Thm}{Theorem}

\newtheorem{Lem}[Thm]{Lemma}

\newtheorem{Prop}[Thm]{Proposition}

\numberwithin{equation}{section}
\numberwithin{Thm}{section}

\theoremstyle{definition}
\newtheorem{Def}[Thm]{Definition}
\newtheorem{Rk}[Thm]{Remark}

\theoremstyle{remark}

\newtheorem*{Thm*}{Theorem}
\newtheorem*{Lem*}{Lemma}
\newtheorem*{Conj*}{Conjecture}
\newtheorem*{Cor*}{Corollary}
\newtheorem*{Def*}{Definition}
\newtheorem*{Prop*}{Proposition}
\newtheorem*{Exo*}{Exercise}
\newtheorem*{Exs*}{Examples}
\newtheorem*{Ex*}{Example}
\newtheorem*{Rk*}{Remark}
\newtheorem*{Rks*}{Remarks}

\begin{document}

\title[On the convexity of injectivity domains]{On the convexity of injectivity domains\\ on nonfocal manifolds}
%on nonfocal manifolds
\author{A. Figalli}
\author{Th. Gallou\"et}
\author{L. Rifford}

\begin{abstract}
Given a smooth nonfocal compact Riemannian manifold, we show that the so-called Ma--Trudinger--Wang condition implies the convexity of injectivity domains. This improves a previous result by Loeper and Villani.\end{abstract}

\maketitle

\section{Introduction}
Let $(M,g)$ be a smooth compact Riemannian manifold of dimension $n \geq  2$. 
The {\em injectivity domain} at a point $x \in M$ is defined as 
\[
\I(x) := \Bigl\{ v \in T_xM \, \vert \, \exists\, t >1 \mbox{ s.t. } d(x,\exp_x(tv)) = |tv|_x \Bigr\},
\]
where $\exp_x$ denotes the exponential mapping at $x$, $d$ the geodesic distance on $M\times M$, and $|v|_x=\sqrt{g_x(v,v)}= \sqrt{\langle v,v\rangle_x}$.
We recall that $I(x)$ is an open star-shaped subset of $T_xM$, and by 
the Itoh-Tanaka Theorem \cite{cr10,it01,ln05} its boundary $\TCL(x)$ (which is called {\em tangent cut locus} at $x$) is Lipschitz.
Its image by the exponential mapping is called the {\em cut locus} of $x$,
$$
\cut(x) := \exp_x \bigl( \TCL(x) \bigr).
$$
Recall that the geodesic distance from $x$, that is the function $y\mapsto d(x,y)$, is smooth outside $\cut(x)$, and more generally the distance function $d$ is smooth outside the set 
$$
\cut(M) := \Bigl\{ (x,y) \in M \times M \, \vert \, y \in \cut(x) \Bigr\}.
$$
For every $x\in M$, $v\in\I(x)$, and $(\xi,\eta)\in T_xM\times T_xM$,
the {\em Ma--Trudinger--Wang tensor} (or MTW tensor for short) at $(x,v)$
evaluated on $(\xi,\eta)$ is defined by the formula
\begin{equation}
\label{eq:MTW s}
\S_{(x,v)}(\xi,\eta) := -\frac32 \, \left. \frac{d^2}{ds^2}\right|_{s=0} 
\left.\frac{d^2}{dt^2}\right|_{t=0}\, \frac{d^2}{2} \Bigl( \exp_x(t\xi),\, \exp_x (v+s\eta)\Bigr).
\end{equation}
(The MTW tensor was introduced for the first time in \cite{mtw05}
in a slightly different way, see also \cite{villaniSF}.)
Since $v \in I(x)$ we have that $\exp_x(v)\not\in \cut(x)$, hence pair of points 
$ \left(\exp_x(t\xi), \exp_x (v+s\eta)\right)$ does not belong to $\cut(M)$ provided $s,t$ are small enough,  and the right-hand side in \eqref{eq:MTW s} is well-defined. As observed by Loeper in \cite{loeper09}, if $\xi, \eta$ are two unit orthogonal vectors in $T_xM$, then
$$
\S_{(x,0)}(\xi,\eta) = \sigma_x(P)
$$
is the sectional curvature of $M$ at $x$ along the plane $P$ generated by $\xi$ and $\eta$.

\begin{Def}\label{def:MTW}
We say that $(M,g)$ satisfies {\bf (MTW)} if the following property is satisfied:
\begin{eqnarray*}
\forall \,x \in M, \, \forall \,v \in \I(x), \, \forall \,\xi,\eta \in T_xM, \quad \Bigl[\<\xi,\eta\>_x = 0 \,\Longrightarrow\, \S_{(x,v)}(\xi,\eta)\geq 0\Bigr].
\end{eqnarray*}
We say that $(M,g)$ satisfies  {\bf (MTW$(K,C)$)} if there exists $(K,C) \in \R \times  \R\cup\{+\infty\}$:
\begin{eqnarray*}
\forall \,x \in M, \, \forall \,v \in \I(x), \, \forall \,\xi,\eta \in T_xM, \quad \S_{(x,v)}(\xi,\eta)\geq -C \left| \langle \xi, \eta\rangle_x \right|
|\xi|_x |\eta|_x +K |\xi|_x^2 |\eta|_x^2.
\end{eqnarray*}
\end{Def}

The {\bf (MTW)} property  imposes hard constraints on the geometry of $(M,g)$. First, by Loeper's observation above, if $(M,g)$ satisfies {\bf (MTW)} then it must have nonnegative sectional curvatures. Moreover, as shown by Loeper and Villani in \cite{lv10},   the  {\bf (MTW)} property has some effects on the geometry of injectivity domains. They proved that if $(M,g)$ is nonfocal and
satisfies a stronger form of the {\bf (MTW)} condition, then all its injectivity domain must be uniformly convex. The aim of the present paper is to improve the result by Loeper and Villani by showing that the strong form of {\bf (MTW)} condition can be dropped. Before to state our main result, let us briefly recall the link between {\bf (MTW)} and the regularity of optimal transports with quadratic geodesic costs, which was the initial motivation for the introduction of the Ma--Trudinger--Wang tensor, see \cite{villaniSF}.\\

Let $\mu,\nu$ be two probability measures on $M$ and $c:M�\times M \rightarrow \R$ be the quadratic geodesic cost defined by
$$
c(x,y):= \frac{d(x,y)^2}{2} \qquad \forall\, (x,y) \in M \times M.
$$
The Monge problem from  $\mu$ to $\nu$ and cost $c$ consists in finding a measurable map $T:M\to M$ which
minimizes the cost functional 
$$
\int_M c(x,T(x))\,d\mu(x)
$$ 
under the constraint $T_\#\mu =\nu$ ($\nu$ is the image measure of $\mu$ by $T$). 
If $\mu$ is absolutely continuous, then according to McCann \cite{mccann01} this minimizing problem has a solution $T$, unique up to modification on a $\mu$-negligible set. 
A natural question is whether the optimal transport map can be expected to be continuous. To this purpose, we introduce the following definition.
\begin{Def}
We say that $(M,g)$ satisfies the {\em transport continuity property} (abbreviated \TCP ) if, whenever $\mu$ and $\nu$ are absolutely continuous measures with respect to the volume measure, with densities bounded away from zero and infinity, the optimal transport map $T$ from $\mu$ to $\nu$ with cost $c$  is continuous, up to modification on a set of zero volume. 
\end{Def}

The following results give necessary and sufficient conditions for  \TCP in terms of the {\bf (MTW)} property and convexity properties of injectivity domains, see \cite{FRV_LAST}. Theirs proofs are based on previous works by many authors, see \cite{figallibourbaki,FigLopdim2,fr09,KimMcC10,Kim:2008ys,loeper09,lv10,mtw05,villaniSF}.

\begin{Thm}
Assume that $(M,g)$ satisfies the \TCP condition. Then $(M,g)$ satisfies {\bf (MTW)} and all its injectivity domains are convex.
\end{Thm}

\begin{Thm}
Assume that $M$ has dimension 2. 
Then  the \TCP condition holds if and only if $(M,g)$ satisfies {\bf (MTW)} and all its injectivity domains are convex.
\end{Thm}

Let us now state our main result. The nonfocal domain at some $x \in M$ is defined as 
$$
\NF(x) := \Bigl\{ v \in T_xM \, \vert \, d_{tv} \exp_x \mbox{ is not singular for any } t \in [0,1]  \Bigr\}.
$$
It is an open star-shaped subset of $T_xM$ whose boundary $\TFL(x)$ is called the tangent focal domain at $x$.
The set $\bar{\NF}(x) = \NF(x) \cup \TFL(x)$ can be shown to be locally semiconvex (see \cite{cr10} and Appendix \ref{App}), and the following inclusion always holds:
$$
\I(x) \subset \NF(x) \qquad \forall\, x\in M,
$$
see for instance \cite[Corollary~3.77]{GHL} or \cite[Problem 8.8]{villaniSF}.

\begin{Def}
We say that $(M,g)$ is nonfocal provided 
$$
\TCL(x) \subset \NF(x) \qquad \forall\, x\in M.
$$ 
\end{Def}

In \cite{lv10}, Loeper and Villani proved that if $(M,g)$ is nonfocal
and satisfies the following strict form of the {\bf (MTW)} condition,
$$
\S_{(x,v)}(\xi,\eta)\geq K |\xi|^2_x |\eta|^2_x \qquad \forall \,x \in M, \, \forall \,v \in \I(x), \, \forall \,\xi, \eta \in T_xM
$$
for some $K>0$, then all its injectivity domain are uniformly convex. Our main result shows that the {\bf (MTW)} condition alone is sufficient for the convexity of injectivity domains.

\begin{Thm}\label{ConNF}
Let $(M,g)$ be a nonfocal Riemannian manifold satisfying {\bf (MTW)}. Then all injectivity domains of $M$ are convex.
\end{Thm}

Our proof is based on techniques relying on the extended Ma--Trudinger--Wang tensor, which were introduced by the first and third author in \cite{fr09}, together with bootstrap arguments. In fact, 
Theorem \ref{ConNF} provides a partial answer to a conjecture formulated by Villani in \cite{villani11}.\\

\noindent {\bf Villani's Conjecture.} 
Let $(M,g)$ be a smooth compact Riemannian manifold satisfying  {\bf (MTW)}. Then all its injectivity domains are convex. \\

We will address the above conjecture in the case of analytic surfaces in a forthcoming paper \cite{FGR2d}. In fact, we take opportunity of the present paper to present a slight improvement (Theorem \ref{ConNFbis})    of Theorem  \ref{ConNF} that will be useful in  \cite{FGR2d}.\\

The paper is structured as follows: In Section \ref{secPREM} we provide some preliminary results about injectivity and nonfocal domains.
Then, Section \ref{secCONV} contains the proof of Theorem \ref{ConNF}. Section \ref{sect:general} is devoted to the proof of Theorem \ref{ConNFbis} whose core of the proof follows the strategy developed in Section $3$ together with additional technicalities, and in Section \ref{sect:conclusion}
we show how to recover Loeper-Villani's result with our techniques.
Finally, in the appendices we collect some useful results on semiconvex functions and tangent cut loci.

\section{Preliminary results} \label{secPREM}
Let $M$ be a smooth compact Riemannian manifold, and denote by $\UM \subset TM$ the unit tangent bundle. Let us introduce some definitions and notation.

The distance function to the cut locus at some $x \in M$, 
$t_{cut} :\UM \rightarrow (0,\infty)$, is defined as
\begin{eqnarray*}
t_{cut}(x,v) & := & \sup \Bigl\{ t \geq 0 \, \vert \, tv \in \I(x) \Bigr\} \\
& = &  \max \Bigl\{ t\geq 0  \, \vert \,   d(x, \exp_x(tv)) = t \Bigr\}.
\end{eqnarray*}
Then, for every $x\in M$, there holds
\begin{multline*}
\I(x)  = \Bigl\{ tv \, \vert \, 0\leq t< t_{cut}(x,v),\ v\in U_xM \Bigr\} ,\qquad  \TCL(x)= \Bigl\{ t_{cut}(x,v)v \, \vert \, v \in U_xM \Bigr\}. 
\end{multline*}
For every $x\in M$, we denote by $\rho_x$ the radial distance on $T_xM$, that is
$$
\rho_x(v,w) := \left\{ \begin{array}{ccl}
 |v|_x + |w|_x & \mbox{ if } & g_x(v,w) \neq |v|_x |w|_x \\
  |v-w|_x & \mbox{ if } & g_x(v,w) = |v|_x |w|_x.
  \end{array}
  \right.
$$
Then the radial distance to $\I(x)$ satisfies for any $v \in T_xM$,
\begin{eqnarray*}
\rho_x \bigl(v,\I(x) \bigr) &:=& \inf \Bigl\{ \rho_x(v,w) \, \vert \, w \in \I(x) \Bigr\}  \\
&=& \left\{
\begin{array}{cl} \left| v - t_{cut}\left(x,\frac{v}{|v|_x}\right) \frac{v}{|v|_x} \right|_x & \mbox{ if } v \notin \I(x),\\
0 & \mbox{ otherwise.}
\end{array}
\right.
\end{eqnarray*}
For every $v \in \TCL(x)$ we set
 $$
 \delta(v) := \max  \Bigl\{ |v-w|_x  \, \vert \, w \in \TCL(x) \, \mbox{ s.t.} \exp_x v = \exp_x w \Bigr\},
 $$
for every compact set $V(x) \subset T_xM$
$$
\delta(V(x)) := \min \Bigl\{ \delta(v)\, \vert \, v \in V(x)\cap \TCL(x) \Bigr\},
$$
and finally for every compact set $V \subset TM$ we let 
$$
\delta(V) :=\min \Bigl\{ \delta \bigl(V(x)\bigr) \, \vert \, x\in M\Bigr\},
$$ 
where for each $x\in M$, $V(x)$ denotes the fiber of $V$ over $x$ (which might be empty,
in which case $\delta(V(x))=+\infty$). Notice that nonfocal compact Riemannian manifolds satisfy $\delta(TM)>0$. However, Riemannian manifolds satisfying $\delta(TM)>0$ are not necessarily nonfocal, as the property $\delta(TM)>0$ only rules out purely focal velocities.

%{\bf j'ai rajouté ici la definition de delta(V) pour un sous ensemble et changer ensuite quand c'était possible l'hypothèse nonfocal par $\delta(V)>0$. Suffisante pour le dernier cas.}

\begin{Lem}\label{lem1}
Let $V$ be a compact subset  of $TM $ with $\delta(V)>0$ such that each $V(x)\neq \emptyset$ is starshaped with respect to the origin. Then, there exists $K >0$ such, that for every $(x,v) \in V$,
$$
 \rho_x \bigl(v,\I(x)\bigr) \leq K \left(  |v|_x^2 - d\bigl(x,\exp_{x}(v)\bigr)^2 \right).
$$
In particular assume that $(M,g)$ is nonfocal. Then, there exists $K >0$ such, that for every $x\in M$ and every $v\in T_xM$,
$$
 \rho_x \bigl(v,\I(x)\bigr) \leq K \left(  |v|_x^2 - d\bigl(x,\exp_{x}(v)\bigr)^2 \right).
$$
\end{Lem}

\begin{proof}[Proof of Lemma \ref{lem1}] 
By compactness of $M$, the geodesic distance (and thus the quantity $d(x,\exp_{x}(v))$) is uniformly bounded. Then since the right-hand side in the inequalities is quadratic in $|v|_x$ while the left-hand size has linear growth, it is sufficient to show that there is $\delta>0$ such that 
$$
 |v|_x^2  - d\bigl(x,\exp_{x}(v)\bigr)^2 \leq  \delta \quad \Longrightarrow \quad \rho_x \bigl(v,\I(x)\bigr) \leq K \left(  |v|_x^2 - d\bigl(x,\exp_{x}(v)\bigr)^2 \right),
$$
for every $(x,v)$ as required. First, for every $(x,v) \in V$ we set $$\psi_x(v):= d_v \exp_x(v) ,$$ so that if $\gamma:[0,1] \rightarrow M$ is a constant-speed minimizing geodesic path going 
from $x$ to $y$, with initial velocity $v_0$ and final velocity $v_1$, the map $\psi_x$ is defined by $v_0 \mapsto v_1$.
Since $\delta(V) > 0$ there exists $\Delta>0$ such that, for every $x\in M$ with $V(x) \neq \emptyset$ and every $v\in V(x) \cap \TCL(x)$, there is a geodesic  path starting at $x$ with initial velocity $w$ (with $|w|_x=|v|_x$), and finishing at $y=\exp_v(x)$ with final velocity $\psi_x(w)$, satisfying
%The nonfocality assumption together with compactness of $V$ implies that there exists $\Delta>0$ such that,
%for every $v\in \TCL(x)$, there is a geodesic starting at $x$ with initial velocity $w$ (with $|w|_x=|v|_x$), and finishing at $y=\exp_v(x)$ with final velocity $\psi_x(w)$, satisfying
\begin{eqnarray}\label{Delta1}
|v|_x^2 - \langle \psi_x(v), \psi_x(w) \rangle_y > \Delta,
\end{eqnarray}
see for instance \cite[Proposition C.5(a)]{lv10}.
Let $v\in \TCL(x)\cap V(x)$ and $y:=\exp_x(v)$ be fixed. 
As before, consider a minimizing geodesic path from $x$ to $y$ with initial velocity $w$ satisfying (\ref{Delta1}).
Since $d^2(x,\cdot)$ is locally semiconcave on $M$, $2\psi_x(w)$ is a supergradient for $d^2(x,\cdot)$ at $y$, and
the distance from $x$ to its cut locus is uniformly bounded from below (see \cite[Definition 10.5 and Proposition 10.15]{villaniSF}),
it is easy to show the existence of a smooth function $h: M \rightarrow \R$, whose $C^2$ norm does not depend on $x$ and $v$, and such that  
$$
\left\{
\begin{array}{l}
d(x,y)^2 = h(y) = |v|_x^2,\\
\nabla h(y) = 2\psi_x(w) \\
d(x,z)^2 \leq  h(z), \ \forall\, z \in M,
\end{array}
\right.
$$
see for instance \cite[Proposition C.6]{lv10}.
This gives 
$$
|(1+\epsilon) v|_x^2 - d\bigl( x, \exp_x((1+\epsilon) v)\bigr)^2  \geq (1+\epsilon)^2 |v|_x^2 - h\bigl( \exp_x((1+\epsilon) v)\bigr) \qquad \forall\, \epsilon.
$$
Hence, if $C_0$ denotes a uniform bound for the $C^2$ norm of $h$ independent of $x$ and $v$, we get
$$
|(1+\epsilon) v|_x^2 - d\bigl( x, \exp_x((1+\epsilon) v)\bigr)^2  \geq 2 \epsilon \left( |v|_x^2 - \langle \psi_x(v),\psi_x(w)\rangle \right) - C_0 \epsilon^2 \qquad \forall \epsilon.
$$
Then, using (\ref{Delta1}), we deduce that
$$
 |(1+\epsilon) v|_x^2 - d\bigl( x, \exp_x((1+\epsilon) v)\bigr)^2  \geq \epsilon \Delta \qquad \forall\, \epsilon \in  (-\epsilon_0,\epsilon_0),
$$
where $\epsilon_0:=\Delta/C_0$.
Since
$$
\rho_x \bigl((1+\epsilon) v, \I(x)\bigr) = | (1+\epsilon) v - v|_x = \epsilon |v|_x,
$$
we finally obtain
$$
 \rho_x \bigl((1+\epsilon) v, \I(x)\bigr) \leq
\frac{|v|_x}{\Delta} \Bigl(   |(1+\epsilon) v|_x^2 - d\bigl( x, \exp_x((1+\epsilon) v)\bigr)^2  \Bigr)   \qquad \forall\, \epsilon \in  (-\epsilon_0,\epsilon_0).
$$
To conclude the proof it suffices to observe that, by a simple compactness argument together with the fact that each $V(x)\neq \emptyset$ is starshaped, one can easily check that
there exists $\delta>0$ such that any $w\in V(x) \setminus \I(x)$, with $ |w|_x^2  - d\bigl(x,\exp_{x}(w)\bigr)^2 \leq  \delta $, has
the form $(1+\epsilon)v$ for some $v\in \TCL(x)\cap V(x)$ and $\epsilon \in [0,\epsilon_0)$.
\end{proof}

\begin{Lem}\label{lem2}
There exists $K >0$ such that for every $(x,v) \in TM$,
$$
K^{-1} \rho_x \bigl(v,\I(x)\bigr) \leq   \rho_y \bigl(w,\I(y)\bigr) \leq K   \rho_x \bigl(v,\I(x)\bigr) 
$$
and for every $(x,v)\in TM$ with $v\in \I(x)$,
$$
K^{-1} \rho_x \bigl(v,\TFL(x)\bigr) \leq   \rho_y \bigl(w,\TFL(y) \bigr) \leq K   \rho_x \bigl(v,\TFL(x)\bigr),
$$
where $y=\exp_x(v)$ and $w = -d_v\exp_x(v)= - \psi_x(v)$, so in particular $x=\exp_y(w)$ . 
\end{Lem}

%{\bf NdTG ce lemme  on avait dit qu'il ne nesessitait pas la nonfocalité, non ? preuve pas contradiction, ou compacite point par point c'est vrai et ouvert juste pour etre sur avant de reecrire. j'ai écrit idem pour NF}
\begin{proof}[Proof of Lemma \ref{lem2}] 
The second inequality follows easily by compactness arguments. Let us prove the first inequality. As before, it is sufficient to show the result provided $\rho_x(v,\I(x))\leq \delta$ for some $\delta>0$. Indeed $\rho_x(v,\I(x))=0$ is equivalent to $\rho_y(w,\I(y))=0$, so all terms vanish.  Let $(x,v) \in TM$ be fixed, set $e_v=\frac{v}{|v|_x}$ and
\[
y=\exp_x(v), \quad w = - \psi_x(v),\quad e_w=\frac{w}{|w|_x}, \quad \bar{w} := t_{cut} \left(y,e_w \right) e_w,
\]
and in addition
\[
\bar{v} := t_{cut} \left(x,e_v \right) e_v, \quad z := \exp_x(\bar{v}), \quad w' :=  -\psi_x(\bar{v}).
\]
Note that since $\bar{v}$ belongs to $\TCL(x)$ the velocity $w'$ belongs to $\TCL(z)$, so it satisfies
\[
w' = t_{cut} \left( z,e_{w'}\right) e_{w'}.
\]
Moreover, 
\[
\rho_x \bigl(v,\I(x)\bigr) = \left| v-\bar{v}\right|_x \quad \mbox{ and }\quad  \rho_y\bigl( w,\I(y)\bigr) =  \left| w-\bar{w}\right|_y.
\]
Equip $TM$ with any distance $d_{TM}$ which in charts is locally bi-Lipschitz equivalent to the Euclidean distance on $\R^n \times \R^n$.  %ndrh es tu sur du terme charts ? 
We may assume that $|v|_x$ is bounded. 
Since the geodesic flow is Lipschitz on compact subsets of $TM$, there holds
\[
d_{TM} \left( (y,w), (z,w')\right) \leq    K' \bigl|v -\bar{v} \bigr|_x,
\]
for some uniform constant $K'$. In fact, if $v$ is close to $\I(x)$ then $\bar{v}$ is close to $v$, and so also $y$ and $z$ are close to each other,
so the above inequality follows from our assumption on $d_{TM}$.
Then, assuming that $\rho_x(v,\I(x)) \leq \delta$ for $\delta >0$ small enough
and taking a local chart in a neighborhood of $y$ if necessary, we may assume that $y, z, w, \bar{w}, w'$ are in $\R^n$.
Moreover, up to a bi-Lipschitz transformation which may affect the estimates only up to a uniform multiplicative constant,
we may assume for simplicity that $d_{TM}$ coincides with the Euclidean distance on $\R^n \times \R^n$.
Since  $y$ is perturbed along the geodesic flow,  Theorem \ref{cutlip} gives  
\begin{align*}
\left| w-\bar{w}\right|_y  &= \left| w \right|_y - t_{cut}(y,e_w) =\left| v \right|_x - \left| \bar v \right|_x + \left| \bar v \right|_x - t_{cut}(y,e_w)  \\
&= \left| v \right|_x-\left| \bar v \right|_x + \left| w' \right|_z - t_{cut}(y,e_w) \\
&=\left| v -\bar v \right|_x + t_{cut}(z,e_{w'})- t_{cut}(y,e_w)\\
&\leq \left| v -\bar v \right|_x + K K' \left| v -\bar v \right|_x.
\end{align*}

\end{proof}

We are now ready to start the proof of Theorem \ref{ConNF}.

\section{Proof of Theorem \ref{ConNF}} \label{secCONV}

Let $(M,g)$ be a smooth compact Riemannian manifold of dimension $n\geq 2$ which is nonfocal and satisfies {\bf (MTW)},
and let $K>0$ be a constant such that all properties of Lemmas \ref{lem1}-\ref{lem2} are satisfied. For every $\mu >0$, we set 
$$
\I^{\mu}(x): = \Bigl\{v \in T_x M \, \vert \, \rho_x(v,\I(x)) \le \mu    \Bigr\}.
$$
Since $M$ is assumed to be nonfocal, there is $\bar{\mu}>0$ small enough such that $\I^{\bar{\mu}}(x)$ does not intersect $\TFL(x)$ for any $x\in M$.

\begin{Lem}\label{Lipcontrol}
Taking $K>0$ larger if necessary, we may assume that for every $x\in M$ and any $v_0,v_1 \in \I(x)$ there holds  
$$
v_t:=(1-t)v_0 +tv_1  \in \I^{K |v_1-v_0|_x}(x)
$$
and 
$$
  \bar{q}_t := - d_{v_t} \exp_x(v_t)   \in \I^{K |v_1-v_0|_x} \bigl( y_t\bigr),
 $$
 with $y_t:=\exp_x (v_t)$.
\end{Lem}
 
\begin{proof}[Proof of Lemma \ref{Lipcontrol}]
Since the functions $v \in U_xM \mapsto t_{cut}(x,v) $ are uniformly Lipschitz, there is $K>0$ such that
$$
\rho_x \bigl(v_t,\I(x) \bigr) \leq K |v_1-v_0|_x \qquad \forall \,v_0,v_1 \in \I(x), \, \forall\, x \in M.
$$
The definition of $\I^{K |v_1-v_0|_x}(x)$ together with Lemma \ref{lem2} yield both inclusions.
\end{proof}
 
Our proof requires the use of the extended MTW tensor which was initially introduced by the first and third author in \cite{fr09}.
To define this extension, we let $x\in M$, $v\in\NF(x)$, and $(\xi,\eta)\in T_xM\times T_xM$.
Since $y:=\exp_xv$ is not conjugate to $x$, by the Inverse Function Theorem there exist an open neighbourhood $\mathcal{V}$ of $(x,v)$ in $\TM$, and an open neighbourhood $\mathcal{W}$ of $(x,y)$ in $M \times M$, such that
$$
\begin{array}{rcl}
\Psi_{(x,v)} : \mathcal{V} \subset TM & \longrightarrow & \mathcal{W} \subset M \times M \\
(x',v') & \longmapsto & \bigl( x',  \exp_{x'}(v') \bigr)
\end{array}
$$
is a smooth diffeomorphism from $\mathcal{V}$ to $\mathcal{W}$. Then we may define
$\hat{c}_{(x,v)}:\mathcal{W} \rightarrow \R$ by
\begin{eqnarray}\label{extendedCOST}
 \hat{c}_{(x,v)} (x',y') := \frac{1}{2} \bigl| \Psi_{(x,v)}^{-1} (x',y') \bigr|_{x'}^2, \qquad
\forall\, (x',y') \in \mathcal{W}.
\end{eqnarray}
If $v\in\I(x)$ then for $y'$ close to $\exp_xv$ and $x'$ close to $x$ we have
$\hat{c}_{(x,v)}(x',y')= c(x',y'):=d(x',y')^2/2$. For every $x\in M$, $v\in\NF(x)$ and $(\xi,\eta)\in T_xM\times T_xM$, the {\em extended Ma--Trudinger--Wang tensor} at $(x,v)$ is defined by the formula

$$
\ov{\S}_{(x,v)}(\xi,\eta)  := -\frac32 \, \left. \frac{d^2}{ds^2}\right|_{s=0}
\left.\frac{d^2}{dt^2}\right|_{t=0}\, \hat{c}_{(x,v)} \Bigl( \exp_x(t\xi),\, \exp_x (v+s\eta)\Bigr).
$$
The following lemma may be seen as an ``extended'' version of \cite[Lemma 2.3]{lv10}.

\begin{Lem}\label{tenseurine}
There exist constants $C, D>0$ such that, for any $(x,v)\in TM$ with $v\in \I^{\bar{\mu}}(x)$, 
$$
\overline{\S}_{(x,v)}(\xi,\eta)\geq -C \left| \langle \xi, \eta\rangle_x \right| |\xi|_x |\eta|_x - D \rho_x(v,\I(x)) |\xi|_x^2 |\eta|_x^2    \qquad \forall\, \xi, \eta \in T_xM.
$$
\end{Lem}

We also give a local version of this theorem when $M$ is not nonfocal.
\begin{Lem}\label{tenseurinebis}
Let $V \subset TM$ and $\mu>0$ such that  
$$
\rho \bigl(V \cap \I,\TFL\bigr) := \sup \Bigl\{ \rho_x(v,w) \, \vert \,  x \in M, \, v \in V(x) \cap \I(x), w \in \TFL(x)\Bigr\} > \mu.
$$
Then there exist constants $C, D>0$ such that, for any $(x,v)\in TM$ with $v\in V(x)\cap \I^{\mu}(x)$, 
$$
\overline{\S}_{(x,v)}(\xi,\eta)\geq -C \left| \langle \xi, \eta\rangle_x \right| |\xi|_x |\eta|_x - D \rho_x(v,\I(x)) |\xi|_x^2 |\eta|_x^2    \qquad \forall\, \xi, \eta \in T_xM.
$$
\end{Lem}
 
\begin{proof}[Proof of Lemma \ref{tenseurine}]
The tensors $\S$ and $\ov{\S}$ coincide on the sets of $(x,v)\in TM$ such that $v\in \I(x)$,
hence
\begin{multline*}
\forall\, (x,v)\in TM \mbox{ with } v \in \I(x), \ \forall\, (\xi,\eta) \in T_xM\times T_xM,\\
\Bigl[\<\xi,\eta\>_x = 0 \,\Longrightarrow\quad \ov{\S}_{(x,v)}(\xi,\eta)\geq 0\Bigr].
\end{multline*}
Let $\I^{\bar{\mu}}(M)$ be the compact subset of $TM$ defined by 
$$
\I^{\bar{\mu}}(M) := \cup_{x\in M} \left( \{x\} \times \I^{\bar{\mu}}(x) \right).
$$
The mapping 
$$
(x,v) \in \I^{\bar{\mu}}(M) \, \longmapsto \, \bigl( x,\exp_x(v)\bigr)
$$
is a smooth local diffeomorphism at any $(x,v) \in \I^{\bar{\mu}}(M)$ and the set of $(x,v,\xi, \eta)$ with  $(x,v) \in \I^{\bar{\mu}}(M)$ and $ \xi, \eta  \in U_xM$ such that $\langle \xi,\eta \rangle_x=0$ is compact. Then there is $D>0$ such that
$$
\ov{\S}_{(x,v)}(\xi,\eta)\geq - D \rho_x(v,\I(x)),
$$
for every  $x,v,\xi, \eta$ with  $(x,v) \in \I^{\bar{\mu}}(M)$ and $ \xi, \eta  \in U_xM$ such that $\langle \xi,\eta \rangle_x=0$. By homogeneity we infer that 
$$
\ov{\S}_{(x,v)}(\xi,\eta)\geq - D \rho_x(v,\I(x)) |\xi|_x^2 |\eta|_x^2,
$$
for every  $x,v,\xi, \eta$ with  $(x,v) \in \I^{\bar{\mu}}(M)$ and $ \xi, \eta  \in T_xM$ such that $\langle \xi,\eta \rangle_x=0$. We conclude as in the proof of \cite[Lemma 2.3]{lv10}.
\end{proof}

The proof of Lemma \ref{tenseurinebis} follows by the same arguments. The following lemma will play a crucial role.

\begin{Lem}\label{LEMineq}
Let $h: [0,1] \rightarrow [0,\infty)$ be a semiconvex function such that $h(0)=h(1)=0$ and let $c \geq 0$ be fixed. Assume that there are $t_1<\ldots < t_N$ in $(0,1)$ such that $h$ is not differentiable at $t_i$ for $i=1,\ldots, N$, is of class $C^2$ on $(0,1) \setminus \{t_1,\ldots, t_N\}$, and satisfies
\begin{eqnarray}\label{LEMineq1}
\ddot{h}(t)\ge -\vert \dot{h}(t) \vert - c \qquad \forall \,t \in [0,1] \setminus \bigl\{t_1,\ldots, t_N \bigr\}.
\end{eqnarray}
Then
\begin{eqnarray}\label{LEMineq2}
h(t)\le c \, t(1-t) \qquad \forall\, t \in [0,1].
\end{eqnarray}
Moreover, if in addition there exists a constant $ \epsilon \ge 0$ such that 
\begin{eqnarray}\label{LEMineq3}
c \le \|h\|_\infty+ \epsilon,
\end{eqnarray}
 then 
\begin{eqnarray}\label{LEMineq4}
 \|h\|_\infty \le \epsilon /3.
\end{eqnarray}
\end{Lem}

\begin{proof}[Proof of Lemma \ref{LEMineq}] 
Let $a>0$ and $f : [0,1] \rightarrow \R$ be the semiconvex function defined by 
$$
f(t)=h(t)-a t(1-t) \qquad \forall \, t \in [0,1].
$$
Let $\bar t $ be a maximum point for $f$. Since $f$ is semiconvex,
it  has to be differentiable at $\bar{t}$, so $\bar{t} \neq t_i$ for $i=1, \ldots, N$. If $\bar t \in(0,1)$,
then there holds $\dot{f} (\bar{t})=0$ and $\ddot{f}(\bar{t}) \le 0$. Thus, using (\ref{LEMineq1}) we get  
$$
|\dot h(\bar{t}) |= a |2 \bar{t} -1 | \leq a,
$$
$$
\quad 0 \geq \ddot{f}(\bar{t})  =  \ddot{h}(\bar{t})+2a  \ge - |\dot h (\bar{t}) | - c + 2a \ge a - c.
$$
This yields a contradiction as soon as $a>c$,
which implies that in that case $f$ attains its maximum on the boundary of $[0,1]$. Since $f(0)=f(1)=0$, we infer that 
$$
h(t)\le a t(1-t) \qquad \forall \, t \in [0,1],
$$
for every $a > c$. Letting $a \downarrow c $, we get (\ref{LEMineq2}). Finally, if
(\ref{LEMineq3}) is satisfied, (\ref{LEMineq2}) implies (recall that $h$ is nonnegative)
$$
\|h\|_\infty  = \sup_{t \in [0,1]} |h(t)|   \le (\|h\|_\infty+\epsilon) \sup_{t \in [0,1]} t(1-t) = (\|h\|_\infty+\epsilon)/4
$$
and  inequality (\ref{LEMineq4}) follows easily. 
\end{proof}

We recall that given $v_0, v_1 \in \I(x)$, for every $t\in [0,1]$ we set
$$
v_t := (1-t)v_0 +tv_1, \qquad y_t:=\exp_x (v_t), \qquad  \bar{q}_t := - d_{v_t} \exp_x(v_t).
$$
In addition, whenever $y_t$ does not belong to $\cut(x)$ (or equivalently $x \notin \cut(y_t)$) we denote by $q_t$ the velocity in $\I(y_t)$ such that
$$
\exp_{y_t} (q_t) = x \quad \mbox{ and } \quad |q_t|_{y_t} = d(x,y_t).
$$
The following results follow respectively from \cite[Lemma B.2]{FRV_LAST} and \cite[Proposition 6.1]{FRV} and do not need the nonfocality assumption. The idea of Lemma \ref{DERh} goes back to Kim and McCann \cite{KimMcC10}.  Lemma \ref{generic} is an improvement of \cite{FigVil08}.

\begin{Lem}\label{generic}
Let $x\in M$ and $v_0,v_1 \in \I(x)$ be fixed. Then,
up to slightly perturbing $v_0$ and $v_1$, we can assume that $v_0, v_1 \in \I(x)$ and that the semiconvex function $h:[0,1] \rightarrow \R$ defined as 
$$
h(t) := \frac{ |v_t|_x^2}{2} - \frac{d( x, y_t)^2}{2} \qquad \forall \, t \in [0,1],
$$
is of class $C^2$ outside a finite set of times $0<t_1< \ldots < t_N<1$ and not differentiable at $t_i$ for $i=1, \ldots, N$.
\end{Lem}

\begin{Lem}\label{DERh}
Let $x\in M$ and $v_0, v_1 \in \I(x)$.
Assume that the function $h$ defined above is $C^2$ outside a finite set of
times $0<t_1< \ldots < t_N<1$, and is not differentiable at $t_i$ for $i=1, \ldots, N$. 
Furthermore, suppose that $[\bar{q}_t,q_t]\subset \NF(y_t)$ for all $t\in [0,1]$.
Then for every $t \in [0,1] \setminus \bigl\{t_1, \ldots, t_N\bigr\}$ we have
\begeq\label{doth}
\dot{h}(t) = \bigl\<q_t-\bar{q}_t, \dot{y}_t\bigr\>_{y_t},
\endeq
\begeq\label{ddoth}
\ddot{h}(t) = \frac23 \int_0^1 (1-s)\,\ov{\S}_{(y_t, (1-s)\bar{q}_t + s q_t)} (\dot{y}_t, q_t-\bar{q}_t)\,ds.
\endeq
\end{Lem}

The next lemma deals with semiconvexity properties of the sets $I(x)$.
We refer the reader to the Appendix \ref{App} for the main definitions and properties of semiconvex sets.

\begin{Lem}\label{lem5}
There exists a large universal constant $K>0$ such that the following properties are satisfied for any $x\in M$: 
\begin{itemize}
\item[(i)] Assume there are constants $\omega >0$ and $\kappa \in (0,\bar{\mu})$ such that  
\begin{eqnarray*}
\forall\, v_0, v_1 \in I(x), \quad |v_1-v_0|_x \leq \omega \quad \Longrightarrow \quad 
\sup_{q \in  [ \bar q_t ,q_t]} \Bigl\{ \rho_{y_t} \bigl(q,\I(y_t) \bigr) \Bigr\} \leq \kappa.
\end{eqnarray*}
Then $\bar{\I}(x)$ is $(K\kappa)$-radial-semiconvex. 
\item[(ii)] Assume there are  constants $ \omega, \alpha,\epsilon \ge 0$ such that 
\begin{multline*}
\forall\, v_0, v_1 \in I(x), \quad |v_1-v_0|_x \leq \omega \quad \\
 \Longrightarrow \quad 
\sup_{q \in  [ \bar q_t ,q_t]} \Bigl\{ \rho_{y_t} \bigl(q,\I(y_t) \bigr) \Bigr\} \leq
\min \biggl\{  \alpha \left( \frac{ |v_t|_x^2}{2} - \frac{d( x, y_t)^2}{2}  \right) + \epsilon, \bar{\mu} \biggr\}.
\end{multline*}
Then $\bar{\I}(x)$ is $(K\epsilon)$-radial-semiconvex.
 \end{itemize}
\end{Lem}

\begin{proof}[Proof of Lemma \ref{lem5}] 
We first prove assertion (i).
We need to show that there is a uniform constant $K>0$ and $\nu >0$ sufficiently small (see Appendix A)
such that, for any $v_0, v_1 \in \I(x)$ with $|v_0-v_1|_x < \nu$,
$$
\rho_x \bigl(v_t,\I(x)\bigr) \leq K \kappa \frac{t(1-t)}{2} \bigl| v_0-v_1\bigr|^2 \qquad \forall \, t \in [0,1].
$$
 As in Lemma \ref{generic} we set
$$
h(t) := \frac{ |v_t|_x^2}{2} - \frac{d( x, y_t)^2}{2} \qquad \forall \, t \in [0,1].
$$
By Lemma \ref{lem1} it is sufficient to show that
$$
h(t) \leq K \kappa \frac{t(1-t)}{2} \bigl| v_0-v_1\bigr|^2 \qquad \forall \, t \in [0,1],
$$
for some constant $K>0$. Let $v_0, v_1 \in \I(x)$ and $\nu>0$ with $ |v_1-v_0|_x <\nu \leq \omega$ be fixed.
By Lemma \ref{generic}, up to slightly perturbing $v_0,v_1$
we may assume that $h:[0,1] \rightarrow \R$ is semiconvex,  $C^2$ outside a finite set of
times $0<t_1< \ldots < t_N<1$, and not differentiable at $t_i$ for $i=1, \ldots, N$.
By Lemmas \ref{tenseurine} and \ref{DERh} (observe that $\kappa < \bar{\mu}$ and
$\I^{\bar{\mu}}(y_t) \subset \NF(y_t)$),
$$
\ddot{h}(t)\ge -C| \dot{h}(t)| |\dot {y_t}|_{y_t}  |{q_t-\bar q_t}|_{y_t} -D \max_{q\in [ \bar q_t ,q_t]} \Bigl\{ \rho_{y_t}(q,\I(y_t))\Bigr\}   |\dot {y_t}|_{y_t}^2| {q_t-\bar q_t}|_{y_t}^2,
$$
for every $t\in [0,1] \setminus \{t_1, \ldots, t_N\}$. Moreover, by compactness of $M$, there is a uniform constant $E>0$ such that
$$
\bigl| \dot{y}_t \bigr|_{y_t} \leq E \bigl| v_0-v_1 \bigr|_x \quad \mbox{ and } \quad  \bigl|q_t-\bar{q}_t \bigr|_{y_t} \leq E.
$$
Hence
\begin{multline}\label{16mars}
\ddot{h}(t) \ge -CE^2 | \dot{h}(t)| \bigl|v_1-v_0\bigr|_x - D E^4 \kappa \bigl|v_1-v_0\bigr|_x^2 \\ \forall \, t\in [0,1] \setminus \bigl\{t_1, \ldots, t_N\bigr\}.
\end{multline}
Taking $\nu \in (0,\omega)$ small enough yields
$$
\ddot{h}(t)\ge -| \dot{h}(t)| - D E^4 \kappa |v_1-v_0|^2
\qquad \forall \, t\in [0,1] \setminus \bigl\{t_1, \ldots, t_N\bigr\},
$$
so Lemma  \ref{LEMineq} gives 
$$
h(t) \leq D E^4 \kappa \, t(1-t)|v_1-v_0|_x^2 \qquad \forall \, t \in [0,1],
$$
which shows that $\I(x)$ is $(K\kappa)$-radial-semiconvex where $K>0$ is a uniform constant.\\
To prove (ii) we note that (\ref{16mars}) implies
\begin{equation}
\label{eq:h ddot}
\ddot{h}(t)\ge -CE^2 | \dot{h}(t)| \bigl|v_1-v_0\bigr|_x - D E^4 \alpha |h(t)| \bigl|v_1-v_0\bigr|_x^2 - DE^4\epsilon \bigl|v_1-v_0\bigr|_x^2,
\end{equation}
which (by choosing $\nu \in (0,\omega) $ sufficiently small) gives
$$
\ddot{h}(t)\ge -| \dot{h}(t)| - \|h\|_{\infty}  - D E^4 \epsilon |v_1-v_0|_x^2 \qquad \forall \, t\in [0,1] \setminus \bigl\{t_1, \ldots, t_N\bigr\}.
$$
Hence, by the second part of Lemma \ref{LEMineq} we obtain
$$
\|h\|_\infty \leq  \frac{D E^4}{3} \epsilon |v_1-v_0|_x^2.
$$
Plugging this information back into \eqref{eq:h ddot} gives, for $\nu$ sufficiently small,
$$
\ddot{h}(t)\ge -CE^2 | \dot{h}(t)| \bigl|v_1-v_0\bigr|_x-2DE^4\epsilon \bigl|v_1-v_0\bigr|_x^2.
$$
We conclude as in the first part of the proof.
 \end{proof}
 
Returning to the proof of Theorem $\ref{ConNF}$, we say that the property $\PP(r)$ is satisfied if for any $x \in M$ the set $B_x(r)\cap \I(x)$ is convex
(here $B_x(r)$ denotes the unit open ball in $T_xM$ with respect to $|\cdot|_x$).
If  $\PP(r)$ is satisfied for any $r\geq 0$, then all the injectivity domains of $M$ are convex.
Since $r_0 := \inf_{x \in M,\,v \in \TCL(x)}  |v|_x$ is strictly positive,
$\PP(r) $ is true for any $r\leq r_0$, hence the set of $r\geq 0$ such that $\PP(r)$ is satisfied is an interval $J$ with positive length.
Moreover, since the convexity property is closed, $J$ is closed. Consequently, in order to prove that $J=[0,\infty)$, it is sufficient to show that $J$ is open.
 
\begin{Lem}{\label{open}}
The set of $r$ for which $\PP(r) $ holds is open in $[0,\infty)$.
\end{Lem}

\begin{proof}[Proof of Lemma \ref{open}]
Assume that $\PP(r)$ holds. We want to prove that, if $\beta>0$ is sufficiently small then $P(r+\beta)$ holds as well. The proof is divided in two steps: first we will show that, for any $\beta \in (0,\bar{\mu}/(2K))$ (here $\mu$
and $K$ are as in Lemma \ref{lem5}),
the sets $B_x(r+\beta)\cap \I(x)$ are $(K\beta)$-radial-semiconvex for any $x\in M$. Then, in Step $2$
we show the following ``bootstrap-type'' result: if the sets  $B_x(r+\beta)\cap \I(x)$ are $A$-{radial-semiconvex} for all $x\in M$, then they are indeed $(A/2)$-{radial-semiconvex}.
The combination of Steps 1 and 2 proves that, for any $x \in M$ and $\beta >0$ small, the sets $B_x(r+\beta)\cap \I(x)$ are $(K\beta/2^k)$-{radial-semiconvex} for any $k \in \N$,
hence convex.
\subsection*{{ Step $1$:}} \textit{$\I(x)\cap B_x(r+\beta)$ is $(K\beta)$-{radial-semiconvex} for any $\beta \in (0,\bar{\mu}/(2K))$.}

Fix $x \in M$ and $\nu>0$. Thanks to Lemma \ref{Lipcontrol}, for any $v_0,v_1\in \I(x)$ with $|v_0 - v_1|_x < \nu$ we have
$$
v_t \in \I^{K \nu }(x) \quad \mbox{ and } \quad \bar q_t \in \I^{K \nu}(y_t).
$$
Let $\beta>0$ and $v_0, v_1\in B_x(r+\beta) \cap \I(x)$ be fixed. By construction
$$
|\bar q_t|_{y_t}=|v_t|_{x} < r+ \beta, \quad | q_t |_{y_t} \le |v_t|_{x} <  r+ \beta, \quad q_t \in \I(y_t).
$$
Since $\bar q_t \in \I^{ K \nu}(y_t)$ we can find  ${q'_t}\in \bar{\I(y_t)}\cap B_{y_t}(r+\beta) $ such that 
$$
\rho_{y_t} \bigl(\bar q_t, \I(y_t)\bigr)=|\bar q_t-{q'_t}| \le K \nu.
$$
\begin{figure}\label{figure1}
\centering
\includegraphics[scale=0.7]{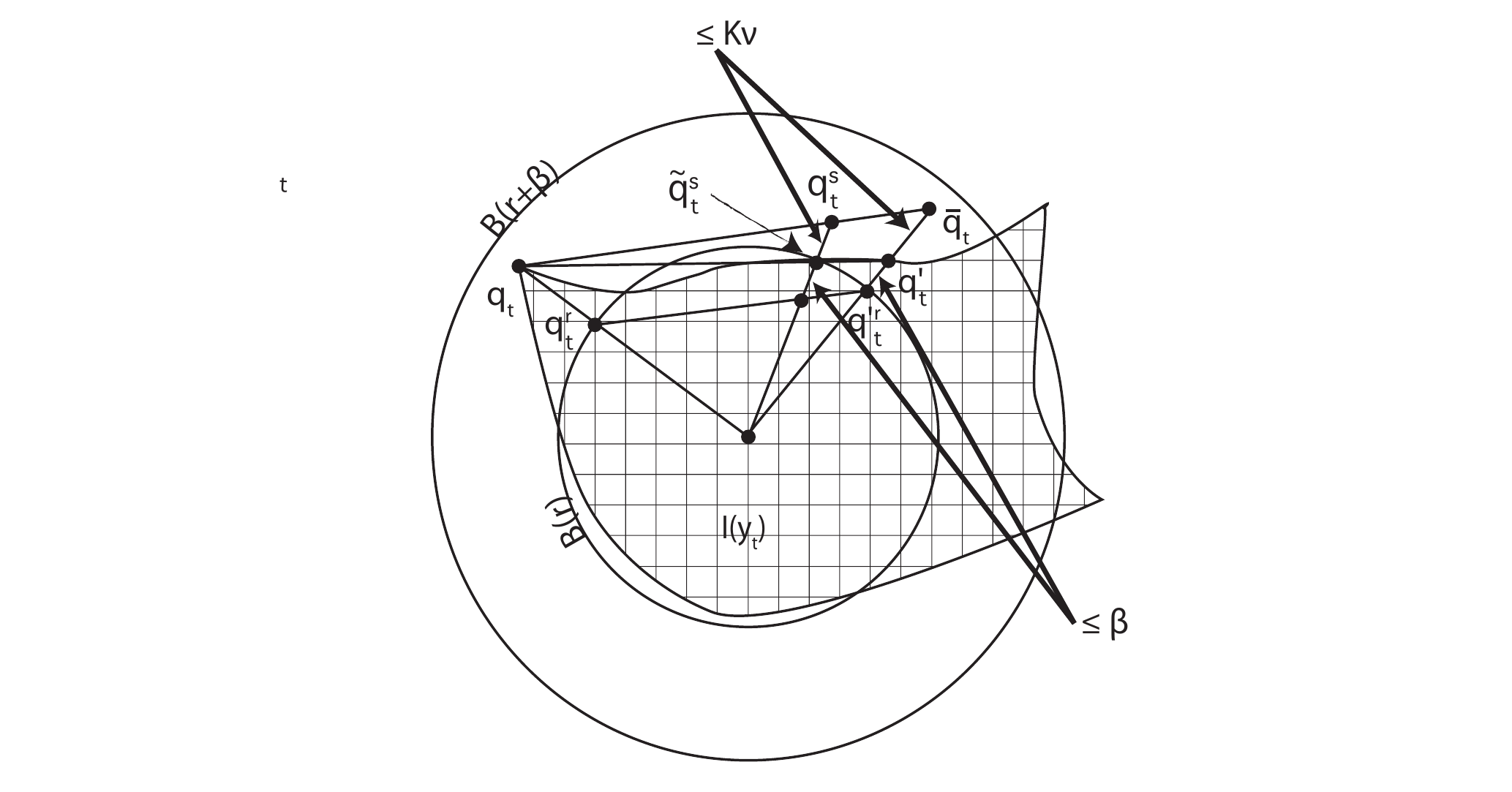}
\caption{Definitions}
\end{figure}
Moreover, using that $\I(y_t)$ is starshaped and that $q_t, q'_t \in B_{y_t}(r+\beta)$, we can find $ q^r_t$, $ {q'}^r_t \in  \bar{B}_{y_t}(r)\cap \bar{\I(y_t)}$
such that $\rho_{y_t} ( q_t, q_t^r) \le \beta$ and  $\rho_{y_t}( {q'_t}, {q'}^r_t) \le \beta$. Recalling that by assumption $\PP(r)$, we have
$[  q^r_t, {q'}^r_t]\subset \bar{\I(y_t)}$, which implies (see Figure $1$)
\begin{eqnarray*}
\max_{q \in [\bar q_t,q_t]} \Bigl\{ \rho_{y_t} \bigl(q, \I(y_t)\bigr) \Bigr\} &  \leq & \max_{q \in [\bar q_t,q_t]} \Bigl\{  \rho_{y_t} \bigl(q,[  q^r_t, {q'}^r_t]\bigr)\Bigr\} \\
& = & \max \Bigl\{   \rho_{y_t} \bigl({q}_t, q^r_t\bigr),  \rho_{y_t} \bigl(\bar q_t, {q'}^r_t\bigr) \Bigr\} \\
& \leq &  \beta +  K \nu,
\end{eqnarray*}
where at the second line we used that the maximum is attained at one of the extrema of the segment.
Thus, Lemma \ref{lem5}(i) gives that $B_x(r+\beta)\cap \I(x)$ is $(K\beta + K^2 \nu)$-semiconvex for any $\beta, \nu >0$ such
that $\beta+K\nu < \bar{\mu}/K$. We conclude by letting $\nu \downarrow 0$.

\subsection*{{ Step $2$:}}\textit{If all $\I(x)\cap B_x(r+\beta)$ are $A$-{radial-semiconvex}, then they are $(A/2)$-{radial-semiconvex}.}

We want to prove that the following holds: there exists $\beta_0>0$ small such that, if for some $A>0$ the sets $\I(x)\cap B_x(r+\beta)$ are $A$-{radial-semiconvex}
for all $x \in M$ and $\beta<\beta_0$, then they are indeed $(A/2)$-{radial-semiconvex}. To this aim, by the results in Appendix \ref{App}, we need to prove that there exists $\nu >0$ sufficiently small such that for every $\beta \in (0,\beta_0)$ ($\beta_0$ to be fixed later, independently of $A$) and
$v_0, v_1 \in B_x(r+\beta) \cap \I(x)$ with $|v_0-v_1|_x < \nu$, we have
$$
\rho_x \bigl(v_t,\I(x)\bigr) \leq \frac{A}{2{K^*}} \frac{t(1-t)}{2} \bigl| v_0-v_1\bigr|^2 \qquad \forall \, t \in [0,1],
$$
where $K^*$ is given by Proposition \ref{equivalence}. Let $v_0, v_1 \in \I(x)$ and $\nu>0$ with $ |v_1-v_0|_x <\nu $, and for
$t, s \in [0,1]$ set $q_t^s := (1-s)\bar q_t +s q_t$ and denote by 
$\tilde{q}_t^s$ the intersection of the segments $[0,q_t^s]$ and $[q_t,q_t']$ (see Figure $1$). 
We have (by Lemmas \ref{lem1} and \ref{lem2})
\begin{eqnarray*}
\rho_{y_t} \bigl(q_t^s, \I(y_t) \bigr) & \leq &  \rho_{y_t} \bigl(q_t^s, \tilde{q}_t^s\bigr)+ \rho_{y_t} \bigl(\tilde{q}_t^s,\I(y_t)\bigr) \\
& \leq &  \rho_{y_t} \bigl(\bar{q}_t, q_t' \bigr) + \rho_{y_t} \bigl(\tilde{q}_t^s,\I(y_t)\bigr) \\
& = &  \rho_{y_t}  \bigl(\bar{q}_t, \I(y_t) \bigr)  + \rho_{y_t} \bigl(\tilde{q}_t^s,\I(y_t)\bigr) \\
& \leq & K \rho_x \bigl(v_t, \I(x)\bigr)+   \rho_{y_t} \bigl(\tilde{q}_t^s,\I(y_t)\bigr) \\
& \leq & K^2 \left( \frac{|v_t|_x^2}{2} - \frac{d(x,y_t)^2}{2} \right) +   \rho_{y_t} \bigl(\tilde{q}_t^s,\I(y_t)\bigr).
 \end{eqnarray*} 
Therefore, for every $t\in [0,1]$ we get
\begin{eqnarray}
\label{18mars1}
\quad \max_{q\in [\bar{q}_t,q_t]} \Bigl\{ \rho_{y_t} \bigl(q, \I(y_t) \bigr) \Bigr\} \leq K^2 \left( \frac{|v_t|_x^2}{2} - \frac{d(x,y_t)^2}{2} \right) + \max_{\hat{q} \in [q_t,q_t']} \Bigl\{ \rho_{y_t} \bigl(\hat{q},\I(y_t)\bigr) \Bigr\}.
\end{eqnarray}
Set for every $t, s\in [0,1], \hat{q}_t^s := (1-s)q_t' + s q_t$. By the $A$-radial-semiconvexity we have
\begin{eqnarray}\label{18mars}
\rho_{y_t} \bigl(\hat{q}_t^s,\I(y_t)\bigr) \leq A \frac{s(1-s)}{2} | q_t -q_t' |_{y_t}^2.
\end{eqnarray}
Then, we finally obtain for $\nu >0$ small enough,
$$
\sup_{q \in [\bar{q}_t,q_t]} \Bigl\{ \rho_{y_t} \bigl(q, \I(y_t) \bigr)\Bigr\} \leq  \min \left\{ K^2 \left( \frac{|v_t|_x^2}{2} - \frac{d(x,y_t)^2}{2} \right) +
A | q_t -q_t' |^2_{y_t}, \bar{\mu} \right\},
$$
for every $t\in [0,1]$. Two cases may appear:

\smallskip

\noindent \textit{First case:} $| q_t-q_t' |_{y_t}^2\le 1/(2K{K^*})$.

In this case, by Lemma \ref{lem5}(ii) we deduce that $\I(x)\cap B_x(r+\beta)$ is $(A/2)$-{radial-semiconvex}.

\smallskip

\noindent \textit{Second case:}  $| q_t-q_t'|_{y_t}^2 > 1/(2K{K^*})$.

\begin{figure}\label{figure2}
\centering
\includegraphics[scale=0.7]{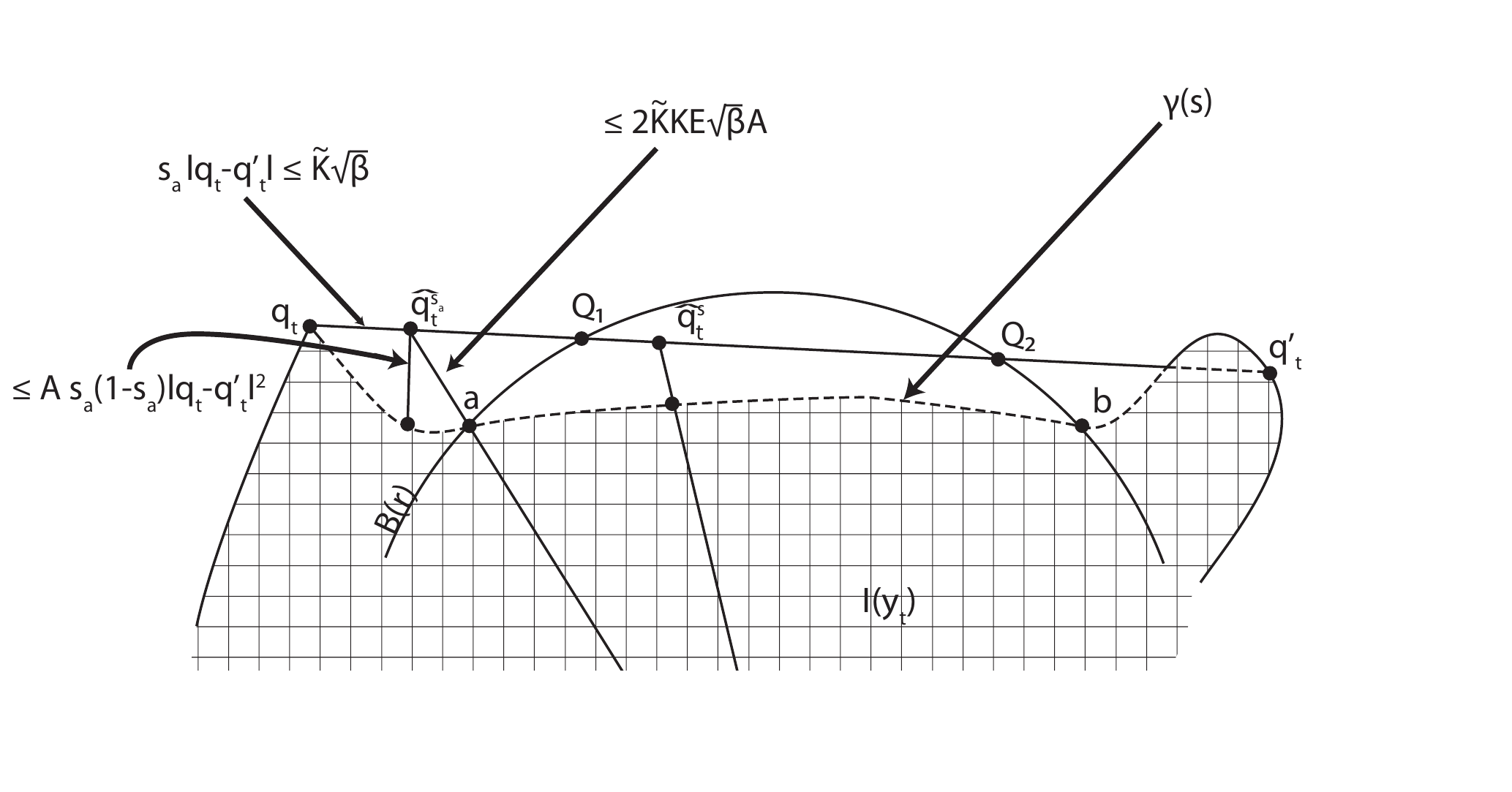}
\caption{Estimations}  
\end{figure}
We work in the plane generated by $0, q_t, q_t'$ in $T_{y_t}M$, and we define the curve $\gamma : [0,1] \rightarrow \I(y_t)$ as (see Figure $2$)
$$
\gamma (s) :=  w \quad \mbox{ where } \quad \rho_{y_t} \bigl( \hat{q}_t^s, \I(y_t)\bigr) = | \hat{q}_t^s-w|_{y_t}  \qquad \forall \,s \in [0,1],
$$
and denote by $a=\gamma(s_a)$ the first point  of $ \gamma$ which enters $ \bar{B}_{y_t}(r)$ and  $b=\gamma( s_b)$ the last one  (see Figure $2$). Since both $q_t, q_t'$ belong to $B_{y_t}(r+\beta)$ and $| q_t-q_t'|_{y_t}^2 > 1/(2K{K^*})$, the intersection of the segment $[q_t,q_t']$ with $B_{y_t}(r)$ is a segment $[Q_1, Q_2]$ such that
$$
\bigl|Q_1 - q_t\bigr|_{y_t}, \, \bigl|Q_2 - q_t'\bigr|_{y_t} \leq \tilde{K}  \sqrt{\beta},
$$
for some uniform constant $\tilde{K}>0$ and $\beta>0$ small enough. Since
$$
\bigl| q_t - \hat{q}_t^{s_a}\bigr|_{y_t} \leq \bigl|Q_1 - q_t\bigr|_{y_t} \quad \mbox{ and } \quad \bigl| q_t' - \hat{q}_t^{s_b}\bigr|_{y_t} \leq \bigl|Q_2 - q_t'\bigr|_{y_t},
$$
this implies that both $s_a$ and $1-s_b$ are bounded by
$\frac{\tilde{K} \sqrt{\beta}}{ |q_t-q_t'|_{y_t}} <\sqrt{2K{K^*} }\tilde{K} \sqrt{\beta}$.
Let us distinguish again two cases:\\
- On $[s_a,s_b]$, $\PP(r)$ is true so $[a,b] \subset \bar{\I(y_t)} $. Hence 
$$
\sup_{q \in [\tilde{q}_t^{s_a}, \tilde{q}_t^{s_b}]} \Bigl\{ \rho_{y_t} \bigl(q, \I(y_t)\bigr) \Bigr\}
\leq  \max \Bigl\{ \rho_{y_t} \bigl(\tilde{q}_t^{s_a}, \I(y_t)\bigr), \rho_{y_t} \bigl(\tilde{q}_t^{s_b}, \I(y_t)\bigr) \Bigr\}.
$$
- On $[0,s_a]$ (similarly on $[1-s_b,1]$), the $A$-{radial-semiconvex}of $\bar{B}_{y_t}(r+\beta)\cap \bar{\I(y_t)}$ yields  (by (\ref{18mars}))
$$
\rho_{y_t} \bigl(\hat{q}_t^s, \I(y_t)\bigr) \leq A \frac{s(1-s)}{2} |q_t-q_t'|_{y_t}^2\le A s_a | q_t-q_t'|_{y_t}^2 \le \sqrt{2K{K^*} }\tilde{K}E \sqrt{\beta} A,
$$
where we used that $ | q_t-q_t'|_{y_t}^2\leq E$ for some uniform constant $E>0$. Recalling (\ref{18mars1}) we obtain
$$
\sup_{q \in [q_t,\bar{q}_t]} \Bigl\{ \rho_{y_t} \bigl(q, \I(y_t) \bigr) \Bigr\} \leq K^2 \left( \frac{|v_t|_x^2}{2} - \frac{d(x,y_t)^2}{2} \right) + \sqrt{2K{K^*} }\tilde{K}E \sqrt{\beta} A.
$$
Hence, if we choose $\beta_0$ sufficiently small so that $\sqrt{2K{K^*} }\tilde{K}E \sqrt{\beta_0} \leq 1/(2K{K^*})$ we get
$$
\sup_{q \in [q_t,\bar{q}_t]} \Bigl\{ \rho_{y_t} \bigl(q, \I(y_t) \bigr) \Bigr\}  \leq K^2 \left( \frac{|v_t|_x^2}{2} - \frac{d(x,y_t)^2}{2} \right) + \frac{A}{2K{K^*}},
$$ 
and we conclude again by Lemma \ref{lem5}(ii).

\smallskip

As explained above, combining Steps 1 and 2 we infer that, for $\beta>0$ small enough, all the $\I(x)\cap B_x(r+\beta)$
are convex. This shows that the interval $J$ is open in $[0,\infty)$, concluding 
 the proof of Lemma \ref{open} and in turn the proof of Theorem \ref{ConNF}.
\end{proof}

As we will see in the next section, we can extract from the proof of Theorem \ref{ConNF} some ideas which will allow us to treat the case of Riemannian manifolds which do not satisfy the nonfocality assumption. Such a result will play a major role in \cite{FGR2d}.
 
\section{General version of the proof of Theorem \ref{ConNF}} 
\label{sect:general}

Let $Z$ be a compact subset in $TM$ whose fibers are denoted by $Z(x)$.
We say that the extended Ma--Trudinger--Wang condition {\bf(}$\bar{{\bf MTW}(-D\rho,C)}${\bf)} holds on
$Z$ if there are constants $C, D>0$ such that, for any $(x,v)\in TM$ with $v\in
Z(x)$, 
$$
\overline{\S}_{(x,v)}(\xi,\eta)\geq -C \left| \langle \xi, \eta\rangle_x \right|
|\xi|_x |\eta|_x - D \rho_x(v,\I(x)) |\xi|_x^2 |\eta|_x^2    \qquad \forall\,
\xi, \eta \in T_xM.
$$
The following improvement of Theorem \ref{ConNF} can be proved by the same method. Note that we do need assume the manifold to be nonfocal.

\begin{Thm}{\label{ConNFbis}}
Let $(M,g)$ be a smooth compact Riemannian manifold, assume that the following property holds: For every $r>0$ such that $B_x(r)\cap \I(x)$ is convex for all $x\in M$, there are $\bar{\beta}(r)>0$ and a compact set $Z\subset TM$ with radial fibers (cf. Definition \ref{def:radial}) satisfying the following properties:
\begin{enumerate}
\item  There are $C,D>0$ such that {\bf(}$\bar{{\bf MTW}(-D\rho,C)}${\bf)} holds on $Z$.
\item There is $K>0$ such that 
$$
\rho_x \bigl(v,\I(x)\bigr) \leq  K \left(  |v|_x^2 -
d\bigl(x,\exp_{x}(v)\bigr)^2 \right) \qquad \forall (x,v)\in Z.
$$
\item  $\forall \,x \in M, \forall \,\beta \in(0,\bar{\beta}(r))$, $\I(x)\cap B_x(r+\beta) \subset Z(x) \subset \bar \NF(x)$.
\item $\forall \,x\in M, \forall \,\beta \in (0,\bar{\beta}(r)), \forall \,v_0, v_1 \in \I(x)\cap B_x(r+\beta)$, $v_t \in Z(x)$ and $[q_t,\bar q_t]
\subset Z(y_t)$.
\end{enumerate}
Then all injectivity domains of $M$ are convex.
\end{Thm}

To prove Theorem \ref{ConNFbis}, we will need the following refined version of Lemma \ref{LEMineq}.

\begin{Lem}\label{LEMineqbis}
Let $h: [0,1] \rightarrow [0,\infty)$ be a semiconvex function such that $h(0)=h(1)=0$ and let $c, C > 0$ be fixed. Assume that there are $t_1<\ldots < t_N$ in $(0,1)$ such that $h$ is not differentiable at $t_i$ for $i=1,\ldots, N$, is of class $C^2$ on $(0,1) \setminus \{t_1,\ldots, t_N\}$, and satisfies
\begin{eqnarray}\label{LEMineq1bis}
\ddot{h}(t)\ge -C\vert \dot{h}(t) \vert - c \qquad \forall \,t \in [0,1]
\setminus \bigl\{t_1,\ldots, t_N \bigr\}.
\end{eqnarray}
Then
\begin{eqnarray}\label{LEMineq2bis}
h(t)\le 4  c e^{(1+C)}  t(1-t) \qquad \forall\, t \in [0,1].
\end{eqnarray}
\end{Lem}

\begin{proof}[Proof of Lemma \ref{LEMineqbis}] 
Given $\mu,\lambda >0$, denote by $f_{\mu,\lambda}:[0,1] \rightarrow \R$  the
semiconvex function defined 
$$
f_{\mu,\lambda}(t) := h(t)- \mu \min \left\{ 1-e^{-\lambda t}, 1-e^{-\lambda
(1-t)}\right\} \qquad \forall \,t \in [0,1].
$$
Let $\bar t $ be a maximum point for $f_{\mu,\lambda}$. Since $f_{\mu,\lambda}$
is semiconvex,
it  has to be differentiable at $\bar{t}$, so $\bar{t} \neq 1/2$ and $\bar{t} \neq t_i$ for $i=1, \ldots, N$. If $\bar t \in(0,1/2)$,
then there holds $\dot{f}_{\mu,\lambda} (\bar{t})=0$ and
$\ddot{f}_{\mu,\lambda}(\bar{t}) \le 0$. Then using (\ref{LEMineq1bis}), we get
 
$$
|\dot h(\bar{t}) |= \mu \lambda e^{-\lambda \bar{t}}, 
$$
$$
0 \geq \ddot{f}_{\mu,\lambda}(\bar{t})  =  \ddot{h}(\bar{t}) + \mu \lambda^2
e^{-\lambda \bar{t}}  \ge -C |\dot h (\bar{t}) | -c  +  \mu \lambda^2
e^{-\lambda \bar{t}} \geq  \mu \lambda (\lambda -C ) e^{-\lambda /2} - c.
$$
This yields a contradiction provided we choose $\lambda = 1 + C$ and $\mu = 2c
e^{1+C}/ (1+C)$  and implies that $f_{\mu,\lambda}$ attains its maximum at
$t=0$. Repeating the same argument on $[1/2,1]$, since  $f(0)=f(1)=0$ we infer
that 
$$
h(t) \leq 2c e^{(1+C)}  \min \left\{ \frac{1-e^{- (1+C)t} }{1+C}, \frac{1-e^{-(1+C) (1-t)}}{1+C} \right\}     \qquad \forall \,t \in [0,1].
$$ 
Noting that 
$$
\frac{1-e^{-(1+C)t}}{1+C} \leq t \quad \mbox{ and } \quad \min \{t,1-t\} \leq 2 t(1-t) \qquad \forall \,t \in [0,1],
$$
we get the result.
\end{proof}
We are ready to give the proof of Theorem \ref{ConNFbis}.

\begin{proof}[Proof of Theorem  \ref{ConNFbis}]
Let $x\in M$ and $v_0,v_1 \in \I(x)$ be fixed. We keep the same notation as in Section \ref{secCONV}.

 The following result is a
variant of Lemma \ref{lem5}. 
\begin{Lem}\label{lem5bis}
Let $r>0$ be such that $B_x(r)\cap \I(x)$ is convex for all $x\in M$ and $\bar \beta(r)$ given by the hypothesis in Theorem \ref{ConNFbis}. There exist  $\bar K$, 
% \bar \kappa >0$ 
such that %the following property is satisfied for every $ \kappa \in (0,\bar
%\kappa)$: 
 if 
$$
\sup_{q \in  [ \bar q_t ,q_t]} \Bigl\{ \rho_{y_t} \bigl(q,\I(y_t) \bigr) \Bigr\} \leq \kappa \qquad \forall \,x\in M, \, \forall \,\beta \in(0,\bar{\beta}(r)), \, \forall \,v_0, v_1
\in \I(x)\cap B_x(r+\beta),
$$
then $\I(x)\cap B_x(r+\beta)$ is $(\kappa\bar K)${-radial-semiconvex}. 
\end{Lem}

\begin{proof}[Proof of Lemma \ref{lem5bis}] 
We need to show that, for any $v_0, v_1 \in \I(x)\cap  B_x(r+\beta)$, 
$$
\rho_x \Bigl(v_t,\I(x)\cap B_x(r+\beta)\Bigr) \leq \kappa \bar K \frac{t(1-t)}{2} \bigl|
v_0-v_1\bigr|^2 \qquad \forall \, t \in [0,1].
$$
 As in Lemma \ref{lem5} we set
$$
h(t) := \frac{ |v_t|_x^2}{2} - \frac{d( x, y_t)^2}{2} \qquad \forall \, t \in
[0,1]
$$
with $v_0, v_1 \in \I(x)\cap B_x(r+\beta)$, and up to slightly
perturbing $v_0,v_1$\
we may assume that $h:[0,1] \rightarrow \R$ is semiconvex,  $C^2$ outside a
finite set of
times $0<t_1< \ldots < t_N<1$, and not differentiable at $t_i$ for $i=1, \ldots,
N$. Moreover properties (1) and (3)-(4) in Theorem  \ref{ConNFbis} yield 
$$
\ddot{h}(t)\ge -C| \dot{h}(t)| |\dot {y_t}|_{y_t}  |{q_t-\bar q_t}|_{y_t} -D
\max_{q\in [ \bar q_t ,q_t]} \Bigl\{ \rho_{y_t}(q,\I(y_t))\Bigr\}   |\dot
{y_t}|_{y_t}^2| {q_t-\bar q_t}|_{y_t}^2,
$$
for every $t\in [0,1] \setminus \{t_1, \ldots, t_N\}$. Since  by compactness of $M$, there is a uniform constant $E>0$ such that
$$
\bigl| \dot{y}_t \bigr|_{y_t} \leq E \bigl| v_0-v_1 \bigr|_x \quad \mbox{ and }
\quad  \bigl|q_t-\bar{q}_t \bigr|_{y_t} \leq E,
$$
we get 
\begin{eqnarray*}
\ddot{h}(t)\ge -CE^2 | \dot{h}(t)| \bigl|v_1-v_0\bigr|_x - D E^4 \kappa
\bigl|v_1-v_0\bigr|_x^2 \qquad \forall \, t\in [0,1] \setminus \bigl\{t_1, \ldots, t_N\bigr\}.
\end{eqnarray*}
Thus Lemma  \ref{LEMineqbis} gives 
$$
h(t) \leq 4e^{(1+CE^2)} D E^4 \kappa \, t(1-t)|v_1-v_0|_x^2 \qquad \forall \, t
\in [0,1],
$$
and so by property (2) in Theorem  \ref{ConNFbis} we get that $\I(x)\cap B(r)$
is $(\kappa \bar K)${-radial-semiconvex} with $\bar K = 2K4e^{(1+CE^2)} D E^4$.%and  $\bar
%\kappa =\delta \frac{K^{*}}{\bar K}$.
\end{proof}

We are ready to apply our bootstrap arguments.  We recall that the property $\PP(r)$ is satisfied if for any $x \in M$ the set $B_x(r)\cap \I(x)$ is convex. As before, in order to conclude the proof of Theorem  \ref{ConNFbis} we just need to prove the following result.
 
\begin{Lem}{\label{openbis}}
The set of $r$ for which $\PP(r) $ holds is open in $[0,\infty)$.
\end{Lem}
 
\begin{proof}[Proof of Lemma \ref{openbis}]
Assume that $\PP(r)$ holds. The proof is divided in two steps: first we show that there are $\beta_0, K>0$ such that, for any $\beta \in (0,\beta_0)$,
the sets $B_x(r+\beta)\cap \I(x)$ are $( (K+1) \bar K \beta)${-radial-semiconvex} for any
$x\in M$. Then in Step $2$
we show the following "bootstrap-type" result: if the sets  $B_x(r+\beta)\cap
\I(x)$ are $A${-radial-semiconvex} for all $x\in M$, then they are indeed
$(A/2)$-{radial-semiconvex}. As before the combination of Steps 1 and 2 proves the convexity of the $B_x(r+\beta)\cap
\I(x)$.

\subsection*{{ Step $1$:}} \textit{$\I(x)\cap B_x(r+\beta)$ is $((K+1) \bar K\beta)$-{radial-semiconvex} for any $\beta \in (0,\beta_0$).}

Fix $x \in M$ and $\beta \in (0,\bar{\beta}(r))$. Since $B_x(r+\beta)\cap
\I(x)$ is starshaped we can find $v_0^r,v_1^r \in \I(x)\cap B(r) $ with, for $i=(0,1)$, $\rho_x(v_i,v_i^r)\leq \beta$. Thus $\PP(r)$ implies that $\rho_x(v_t,I(x))\leq \beta$ for all $t\in [0,1]$, that is $v_t \in \I^{\beta}(x)$, and it follows from Lemma \ref{lem2} that 
$\bar q_t \in \I^{K \beta}(y_t)$. By construction we also have
$$
|\bar q_t|_{y_t}=|v_t|_{x} < r+ \beta, \quad | q_t |_{y_t} \le |v_t|_{x} <  r+
\beta, \quad q_t \in \I(y_t).
$$
Since $\bar q_t \in \I^{ K \beta}(y_t)$ we can find  ${q'_t}\in
\bar{\I(y_t)}\cap B_{y_t}(r+\beta) $ such that 
$$
\rho_{y_t} \bigl(\bar q_t, \I(y_t)\bigr)=|\bar q_t-{q'_t}| \le K \beta.
$$
Moreover, using that $\I(y_t)$ is starshaped and that $q_t, q'_t \in
B_{y_t}(r+\beta)$, we can find $ q^r_t$, $ {q'}^r_t \in  \bar{B}_{y_t}(r)\cap
\bar{\I(y_t)}$
such that $\rho_{y_t} ( q_t, q_t^r) \le \beta$ and  $\rho_{y_t}( {q'_t},
{q'}^r_t) \le \beta$. Again $\PP(r)$ implies that
$[  q^r_t, {q'}^r_t]\subset \bar{\I(y_t)}$, so (see Figure $1$)
\begin{eqnarray*}\label{parhasard}
\max_{q \in [\bar q_t,q_t]} \Bigl\{ \rho_{y_t} \bigl(q, \I(y_t)\bigr) \Bigr\} &
 \leq & \max_{q \in [\bar q_t,q_t]} \Bigl\{  \rho_{y_t} \bigl(q,[  q^r_t,
{q'}^r_t]\bigr)\Bigr\} \\
& = & \max \Bigl\{   \rho_{y_t} \bigl({q}_t, q^r_t\bigr),  \rho_{y_t} \bigl(\bar
q_t, {q'}^r_t\bigr) \Bigr\} \\
& \leq &  \beta +  K \beta,
\end{eqnarray*}
where at the second line we used that the maximum is attained at one of the
extrema of the segment.
Thus, Lemma \ref{lem5bis} implies that $B_x(r+\beta)\cap \I(x)$ is
$((K+1) \bar K \beta )$-{radial-semiconvex} for any $\beta \in$ $]0,\bar{\beta}(r)]$.% with $\bar{\beta}(r)> \beta >0$.% ($\beta <  ) $.% with $\beta < :=\min(\frac{\bar \kappa}{(K+1)}, \bar{\beta}(r) )$. 
\subsection*{{ Step $2$:}} \textit{If all $\I(x)\cap B_x(r+\beta)$ are $A$-{radial-semiconvex}, then they are $(A/2)$-{radial-semiconvex}.}

Let $v_0, v_1 \in \I(x)\cap B_x(r+\beta)$, as before we define in the plane generated by $0, v_0, v_1$ in $T_{x}M$ the curve
$\gamma : [0,1] \rightarrow \I(x)$ by (see Figure $2$)
$$
\gamma (t) =  w \quad \mbox{ where } \quad \rho_{x} \bigl( v_t, \I(x)\bigr) = |
v_t-w|_{x}  \qquad \forall \,t \in [0,1],
$$
and denote by $a=\gamma(t_a)$ the first point  of $ \gamma$ which enters $\bar{B}_{x}(r)$ and  $b=\gamma( t_b)$ the last one. Since both $v_0, v_1$ belong to $B_{x}(r+\beta)$ and $B_{x}(r)\cap \I(x)$ is
convex, the intersection of the segment $[v_0, v_1]$ with $B_{x}(r)$ is a
segment $[Q_1, Q_2]$ such that
$$
\bigl|Q_1 - v_0\bigr|,  \, \bigl|Q_2 - v_1\bigr|_{x} \leq \tilde{K} 
\sqrt{\beta},
$$
for some uniform constant $\tilde{K}>0$ and $\beta>0$ small enough. Since
$$
\bigl| v_{t_a}-v_0 \bigr|_{x} \leq \bigl|Q_1 - v_0\bigr|_{x} \quad \mbox{ and }
\quad \bigl| v_{t_b}-v_1 \bigr|_{x} \leq \bigl|Q_2 - v_1\bigr|_{x},
$$
both $t_a$ and $1-t_b$ are bounded by
$\frac{\tilde{K} \sqrt{\beta}}{ |v_0 -v_1 |_{x}} $.
Let us distinguish two cases:\\
- On $[t_a,t_b]$, $\PP(r)$ is true so $[a,b] \subset \bar{\I(x)} $. Then
$$
\sup_{v \in [v_{t_a}, v_{t_b}]} \Bigl\{ \rho_{x} \bigl(v, \I(x)\bigr) \Bigr\}
\leq  \max \Bigl\{ \rho_{x} \bigl(v_{t_a}, \I(x)\bigr), \rho_{x} \bigl(v_{t_b},
\I(x)\bigr) \Bigr\}.
$$
- On $[0,t_a]$ (similarly on $[1-t_b,1]$),  $\bar{B}_{x}(r+\beta)\cap
\bar{\I(x)}$ is $A$-{radial-semiconvex}, so
$$
\rho_{x} \bigl(v_t, \I(x)\bigr) \leq A  \frac{t(1-t)}{2}|v_1-v_0 
|^2_{x}\le AE\tilde{K} \sqrt{\beta}.
$$
Combining these two estimates we get, for all $t\in [0,1]$,
$$
\rho_{x} \bigl(v_t, \I(x)\bigr) \leq A  \frac{t(1-t)}{2}|v_1-v_0 
|^2_{x}\le AE\tilde{K} \sqrt{\beta}.
$$
Then we define as above ${q}_t'$ such that $\rho_{y_t} \bigl(\bar q_t,
\I(y_t)\bigr)=|\bar q_t-{q}_t'|$. By Lemma \ref{lem2} we get
\begin{eqnarray*}
\max_{q \in [\bar q_t,q_t]} \Bigl\{ \rho_{y_t} \bigl(q, \I(y_t)\bigr) \Bigr\} &
 \leq &\rho_{y_t} \bigl(\bar q_t,
\I(y_t)\bigr)+ \max_{\hat{q}\in [{q}_t',q_t]} \Bigl\{  \rho_{y_t}
\bigl(\hat{q},\I(y_t) \bigr)\Bigr\} \\
& \leq & KAE\tilde{K} \sqrt{\beta} + \max_{\hat{q} \in
[q_t',q_t]} \Bigl\{ \rho_{y_t} \bigl(\hat{q},\I(y_t)\bigr) \Bigr\}.
\end{eqnarray*}
Since $\bar{B}_{x}(r+\beta)\cap \bar{\I(x)}$ is $A$-{radial-semiconvex} for every $x \in M$,
the same argument used above for $[v_0,v_1]$ is also valid on each segment $ [q_t',q_t]$, hence
$$
\quad \max_{q\in [\bar{q}_t,q_t]} \Bigl\{ \rho_{y_t} \bigl(q, \I(y_t) \bigr)
\Bigr\} \leq KAE\tilde{K} \sqrt{\beta} + AE\tilde{K} \sqrt{\beta}
$$
Therefore, if we choose $\bar{\beta}(r)>0$ sufficiently small  we get
$$
\sup_{q \in [q_t,\bar{q}_t]} \left\{ \rho_{y_t} \bigl(q, \I(y_t) \bigr) \right\}
 \leq \frac{A}{2\bar K},
$$ 
and in turn, by Lemma \ref{lem5bis},
$$
\rho_x\bigl( v_t, \I(x)\bigr) \leq \frac{A}{2} \frac{t(1-t)}{2} \bigl| v_0-v_1\bigr|_x^2 \qquad \forall \,t \in [0,1],
$$
which proves the $(A/2)$-{radial-semiconvex}.
\end{proof}
The proof of Lemma \ref{openbis} concludes the proof of Theorem \ref{ConNFbis}.
\end{proof}

We leave the reader to check that if $(M,g)$ is nonfocal, then the properties in Theorem  \ref{ConNFbis} are satisfied (take $Z= \I^{\bar \mu}$ which was defined in Section \ref{secCONV}). As a consequence, Theorem \ref{ConNF} can be seen as a corollary of Theorem \ref{ConNFbis}. 

\section{Conclusion and perspectives}
\label{sect:conclusion}

We can develop our proof further to cover all the results obtained in \cite{lv10}, namely modifying just a bit Lemma \ref{LEMineqbis} we can prove that  {\bf (MTW$(\kappa_0,\infty)$)} for $\kappa_0>0$ gives $\kappa$-uniform convexity for some $\kappa >0$. For a definition of $\kappa$ uniform convexity we refer to \cite{lv10} or Appendix A.

%{\bf AF:  {\bf (MTW$(\kappa_0)$)} n'a pas ete defini...on pourrait ecrire  {\bf (MTW$(\kappa_0,\infty)$)},
%non?}

\begin{Lem}\label{LEMineqbism}[Modified lemma]
Let $h: [0,1] \rightarrow [0,\infty)$ be a semiconvex function such that $h(0)=h(1)=0$ and let $c, C > 0$ be two fixed constants.
Assume that there are $t_1<\ldots < t_N$ in $(0,1)$ such that $h$ is not differentiable at $t_i$ for $i=1,\ldots, N$, is of class $C^2$ on $(0,1) \setminus \{t_1,\ldots, t_N\}$, and satisfies
\begin{eqnarray}\label{LEMineq1bism}
\ddot{h}(t)\ge  -C\vert \dot{h}(t) \vert + c \qquad \forall \,t \in [0,1]
\setminus \bigl\{t_1,\ldots, t_N \bigr\}.
\end{eqnarray}
Then
\begin{eqnarray}\label{LEMineq2bism}
h(t)\le- 4  c e^{(1+C)}  t(1-t) \qquad \forall\, t \in [0,1].
\end{eqnarray}
\end{Lem}
It leads to the following theorem.
 \begin{Thm}{\label{ConNFm}}
Let $(M,g)$ be a nonfocal Riemannian manifold satisfying {\bf (MTW$(\kappa_0,\infty)$)}, with $\kappa_0>0$. Then there exists $\kappa>0$ such that all injectivity domains of $M$ are $\kappa$ uniformly convex.
\end{Thm}
\begin{proof}[Sketch of the proof ]%of Theorem \ref{ConNFm}
Thanks to Theorem \ref{ConNF} we know that for all $x \in M$, $\I(x)$ is convex. Therefore we can define 
$\forall v_0,v_1 \in \partial \I(x)$, $v_t=(1-t)v_0+tv_1\in \I(x)$, $q_t=\exp_x\left(t_c(v_t)v_t\right)$ and
$$
h(t) := \frac{ |v_t|_x^2}{2} - \frac{d( x, y_t)^2}{2} \qquad \forall \, t \in
[0,1].
$$
According to \cite{lv10} we deduce from {\bf (MTW$(\kappa_0,\infty)$)} that $M$ satisfies {\bf (MTW$(\kappa_0,C)$)}, where $C>0$.
We conclude thanks to Lemmas \ref{DERh} and \ref{LEMineqbism}.
\end{proof} 
%\tgdeux{With this lemma we indeed obtain that the path $t\mapsto v_t$ is strictly inside $I(x)$, and we get a lower bound of the curvature of $\TCL(x)$.
%The rest of the proof goes on easily since, once the convexity is proved, there is no more problem in defining the paths 

%{\bf AF: peut etre ici on devrait expliquer un peu mieux qu'est ce que ca veut dire ``definir les chemins'' et aussi expliquer un peu mieux la preuve...juste 2-3 lignes de plus} and we can run the exact same argument as we did previously using Lemma \ref{LEMineq1bism}. This gives for example to the following theorem.}

 Theorem \ref{ConNFbis} is very general, it can be extended to $\kappa$ uniform convexity.
 %\tgdeux{ and let us the possibility to use it for the nonfocal case in any dimension. {\bf AF: je ne comprend pas 
 %ce commentaire. On avait deja dit a la fin de la section 4 que on pouvait utilizer le theoreme 4.1
 %dans le cas nonfocal....} }
 We only need to find a domain satisfying the control  condition $(1)$--$(4)$ of Theorem \ref{ConNFbis}. For this construction we face two difficulties located around the purely focal points, the first one is to give a sign to the extended tensor near these points, The second one is to isolate them. To be done we need to better understand the repartition of purely focal points, and the behavior of the tensor near them. We adopt this strategy for an analytic manifold of dimension $2$ in \cite{FGR2d}. If one succeed in proving the Villani's conjecture, it will give a very nice formulation of necessary and sufficient conditions for regularity of optimal transport maps \cite{FRV_LAST}.

  %it will give a very nice formulation for the \TCP  condition. In $\R^n$ we can go further and obtain more regularity on the transport map  \cite{mtw05, TruWan09, Liu10}.
% {\bf NdTG y'a pas un article de Kim Macann avec de la r\'egularit\'e ? $\R^n$ ou variete ? }
 
 % for a $\mathcal{TC^{\infty}P}$ condition 
%To do this on a Riemannian manifold one also need to catch the behavior near purely focal points or ovoid them as in \cite{lv10}.  %example $C^{\infty}$
%{\bf AF: on a pas definit $\mathcal{TC^{\infty}P}$, meme si je comprend ce que on veut dire...peut etre mieux
%enlever $\mathcal{TC^{\infty}P}$ et mettre juste des mots pour expliquer}

%\begin{appendices}
%\pagestyle{styleannA}
%\clearemptydoublepage
%\pagestyle{styleannA}
%\newpage

%\chapter{Reminders} \label{annexA}
\begin{appendices}
\section{Semiconvexity}\label{App}
%\addcontentsline{toc}{chapter}{Appendix A}

Following \cite{lv10} we recall several equivalent definitions for semiconvex functions.

\begin{Def}[Semiconvexity]
Let $O$ be a convex subset of $\R^n$. A function $f$ : $O  \rightarrow \R$ is said to be $\delta$-semiconvex if
equivalently, for any $x$, $y$ in $\R^n$ and $t$ in $[0,1]$,
\begin{enumerate}[(i)]
\item  $f((1-t)x+t(y))\le (1-t)f(x)+tf(y) + \delta t(1-t)\frac{|x-y|^2}{2};$
\item $f+\delta \frac{|x|^2}{2},$ is convex;
\item $\nabla^2 f \ge -\delta.$
\end{enumerate}
\end{Def}
Here (iii) has to be understood in a distributional sense where $f$ is not differentiable. The equivalent of $(i)$, $(ii)$, and $(iii)$ is a classical convexity result. Note that $(iii)$ tells us that as convexity, semiconvexity may be seen as a local property. When $\delta<0$ we find the uniform convexity.

\begin{Def}\label{def:radial}
An open set $V \subset \R^{n+1}$ is a Lipschitz radial set if it is  starshaped around $0$ and its boundary is
% parametrized by a 
Lipschitz.
%function defined on $\Sph^n$. %(the unit sphere centered at the origin). 
\end{Def}
Here and in the sequel,
$\rho$ denotes the radial distance as defined in Section \ref{secPREM}.

\begin{Def}\label{def:semiconvex}
A radial set $V$ is said to be
\begin{itemize}
\item {\bf$\delta$-distance-semiconvex }if $\dist(\cdot,\bar V)$ is $\delta$-semiconvex, that is for any $x,y\in V$, the function $h(t):=\dist((1-t)x+ty,\bar
V)$ is $\delta$-semiconvex on $[0,1]$.  
\item {\bf locally $\delta$-distance-semiconvex} if there exists $\nu>0$ such that for any $x,y \in V$ with $|x-y|<\nu$, the function $h(t):=\dist((1-t)x+ty,\bar
V)$ is $\delta$-semiconvex on $[0,1]$.
\item {\bf $\delta$-radial-semiconvex} if $\rho$ is $\delta$ semiconvex, that is for any $x,y \in V$ the function $h(t):=\rho((1-t)x+ty,\bar
V)$ is $\delta$-semiconvex on $[0,1]$.
\item {\bf locally $\delta$-radial-semiconvex} if there exists $\nu>0$ such that for any $x,y \in V$ with $|x-y|<\nu$, the function $h(t):=\rho((1-t)x+ty,\bar
V)$ is $\delta$-semiconvex on $[0,1]$.
\end{itemize}
\end{Def}
These definitions are very much inspired by the definition of $\kappa$--uniform convexity in \cite{lv10}. To obtain both notions in one definition 
one need to consider the signed distance function $\dist_{sign}(\cdot,\partial V)$ (resp. $\rho_{sign}(\cdot,\partial V)$) instead of $\dist(\cdot,\bar V)$ (resp. $\rho$): we take the distance with the negative sign when we are inside $V$.

%Note that any Lipschitz radial set is starshaped with respect to the origin. Or $dist(x,\bar V)$ is semiconvex. In this case $V$ will be a sublevel of a semiconvex function define on
%$\R^{n+1}$.

%{\bf NdLR: ou est le $\delta$ dans le $h$ qui apparait pour la definition de $\delta$-distance-semiconvex ? J'ai ajoute lipschitz car ca me semble important. Mets au propre la remarque ci-dessus qui donne une definition equivalent en terme de  $dist(x,\bar V)$ et sur cette histoire de sublevel set.}\\

%For a general semiconvex set (define as a sublevel) the second fundamental form
%as to be bounded below by $-\kappa$, but in general this bound does not imply
%the $\kappa$ semiconvexity of the distance as in the uniformly convex case. It
%also depends on the global shape of the domain. The $\kappa $ semiconvexity
%holds only locally.

%But in the case of a domain with $S^n$ lip boundary the result is true.

%For a radial set  we also define $\rho_0$ the radial distance as in section $1$.

%{\bf NdLR: toutes les phrases ci-dessus ne sont pas claires. relis et corrige. C'est quoi un ensemble $\delta$-semiconvex ? Tu as dit ce qu'est un ensemble $\delta$-distance-semiconvex mais pas $\delta$-semiconvex. Corriger aussi ci-dessous.}

\begin{Prop}\label{equivalence}
If a radial set $V$ is (locally) $ \delta $-distance-semiconvex then it is (locally) $K^*\delta$-radial-semiconvex for some $K^*>0$. Reciprocally if $V$ is (locally) $\delta$-radial-semiconvex then it is (locally) $\delta$-distance-semiconvex.
\end{Prop}
\begin{proof}
Equation (A.4) of \cite{lv10} provides a constant $K^*>0$ depending on the dimension, the Lipschitz regularity, and the diameter of $V$, such that 
$$
\dist(\cdot,\partial V)\leq \rho(\cdot,\partial V) \leq K^* \dist(\cdot,\partial V).
$$
\end{proof}
\begin{Prop}
If a radial set $V$ is $0$-radial-semiconvex  then it is convex.
\end{Prop}
\begin{proof}
For any $x,y \in V$ we have  $\rho((1-t)x+ty,\bar V) \leq 0$, that is $[x,y]\in \bar V$.
\end{proof}

%\begin{Prop}
%If a radial set $V$ is $ \delta $-semiconvex then $\rho_o$ (the radial
%distance) is $K \delta $-semiconvex for some $K>0$. Reciprocally if $\rho_0$ is $\delta$ semiconvex then $V$ is $\delta$ semiconvex.
%\end{Prop}
%\begin{proof}
%By compactness there is a $\eta>0$ such that the angle between the radius and the
%tangent  at a point $x$ in $\partial A$ is bigger than $\eta$. 
%This gives the equivalence of the two distance.
%\end{proof}

\begin{Prop}
Let $V$ be a radial set which is locally $\delta$-distance-semiconvex, then $V$ is $\delta$-distance-semiconvex. 
If  $V$ is locally $\delta$-radial-semiconvex then $V$ is $K^*\delta$-radial-semiconvex, where $K^*$ is given by Proposition \ref{equivalence}.
\end{Prop}
\begin{proof}
The first assertion can be deduced from Proposition A.4 of \cite{lv10}. The second follows from our Proposition \ref{equivalence}
\end{proof}

\section{The tangent cut loci  are Lipschitz continuous}
\subsection{Introduction}
Let $(M,g)$ be a smooth compact Riemannian manifold of dimension $n\geq 2$. 
%The injectivity domain at some  point $x \in M$ is defined as 
%\begin{equation}
%\I(x) = \Bigl\{ v \in T_xM \, \vert \, \exists\, t >1 \mbox{ s.t. } d(x,\exp_x(tv)) = |tv|_x \Bigr\},
%\end{equation}
%or equivalently
%\begin{equation}\label{defI}
%\I(x) = \Bigl\{ v \in T_xM \, \vert \, \exists\, t >1 \mbox{ s.t. } d^2(x,\exp_x(tv)) = |tv|^2_x \Bigr\},
%\end{equation}
%where $\exp_x$ denotes the exponential mapping at $x$, $d$ the geodesic distance on $M\times M$, $|v|_x=\sqrt{g_x(v,v)}= \sqrt{\langle v,v\rangle_x}$ and $d^2$ the squared geodesic distance or quadratic geodesic distance (associated to the Lagrangian $|v|_x^2$). 
%The injectivity domain $I(x)$ is an open star-shaped subset of $T_xM$; is boundary  $\TCL(x)$, which is called the tangent cut locus at $x$,  can be described thanks to a function $t_{cut}$ defined on $U M \subset TM$: for any $(x,v)\in UM$ {\it i.e. }$(x,v)\in TM$ and $||v||=1$, we define $t_{cut}$ by 
%\begin{eqnarray}\label{deftc}
%t_{c}(x,v) & := & \sup \Bigl\{ t \geq 0 \, \vert \, tv \in \I(x) \Bigr\} \\
%& = &  \max \Bigl\{ t\geq 0  \, \vert \,   d^2(x, \exp_x(tv)) = |t|^2_x \Bigr\}.
%\end{eqnarray}
%Then, for every $x\in M$, there holds
%\begin{multline}\label{defIavectc}
%\I(x)  = \Bigl\{ tv \, \vert \, 0\leq t< t_{c}(x,v),\ v\in U_xM \Bigr\} \\
%\mbox{ and } \quad  \TCL(x)= \Bigl\{ t_{c}(x,v)v \, \vert \, v \in U_xM \Bigr\}. 
%\end{multline}
We know that the function $t_{cut}$ defined in Section \ref{secPREM} is bounded from below by the injectivity radius of $M$ and bounded from above by the diameter of $M$.  

%We now define the nonfocal domain at some $x \in M$ as
%\begin{equation}\label{defNF}
%\NF(x) = \Bigl\{ v \in T_xM \, \vert \, d_{tv} \exp_x \mbox{ is not singular for any } t \in [0,1]  \Bigr\}.
%\end{equation}
%It is an open star-shaped subset of $T_xM$ whose boundary $\TFL(x)$ is called the tangent focal domain at $x$ and can be described by the function $t_f$ defined on $UM \subset TM$ by
%\begin{eqnarray}\label{deftf}
%t_{f}(x,v) & := & \sup \Bigl\{ s \geq 0 \, \vert \, sv \in \NF(x) \Bigr\}.
%\end{eqnarray}
%Then, for every $x\in M$, there holds
%\begin{multline}\label{defNFavectf}
%\NF(x)  = \Bigl\{ tv \, \vert \, 0\leq t< t_{f}(x,v),\ v\in U_xM \Bigr\} \\
%\mbox{ and } \quad  \TFL(x)= \Bigl\{ t_{f}(x,v)v \, \vert \, v \in U_xM \Bigr\} . 
%\end{multline}
In the spirit of the definition of $t_{cut}$ and $t_f$ we define, for any subset $O$ of $TM$ with starshaped fibers, the boundary function $t_b$~:~$UM \to \R^+$ by
\begin{eqnarray*}
t_{b}(x,v) :=  \sup \Bigl\{ t \geq 0 \, \vert \, tv \in O_x \Bigr\}
\end{eqnarray*}
We then give  the notion of  $\kappa$--Lipschitz continuity for $O$.
\begin{Def}[$\kappa$--Lipschitz continuity]\label{Llip}
Let $O \subset TM$ be such that, for any $x\in M$, the fiber $O_x$ is starshaped. 
The set $O$ is $\kappa$--Lipschitz continuous if for any $(\bar x, \bar v) \in UM$,
%$\bar x\in M$, and for any $x$ in a neighbourhood of $\bar x$,  $O_x$ be a family of starshaped subset of $T_xM$. $O$ is $\kappa$ Lipschitz continuous if for any $v\in U_xM$ 
there exists a $\kappa$--Lipschitz continuous function $\tau$ defined on a neighbourhood in $UM$ of $(\bar x, \bar v)$ such that $t_{b}(x,v) \le \tau(x,v)$ and $ t_{b}(\bar x, \bar v) = \tau(\bar x, \bar v)$, where $t_{b}$ is the boundary function for $O$.
\end{Def}
This definition implies that the boundary of $O_x$ is locally a $\kappa$--Lipschitz continuous function.  
%\begin{proof}
%The proof is a straightforward contradiction argument.
%\end{proof}
Our aim is to prove the following theorem:
\begin{Thm}[Lipschitz continuity of the tangent cut loci]\label{cutlip}~ \\   %ndrh ce serait bien de mettre un titre au theorem
\vspace{-.4cm}
\begin{enumerate}
\item There exists  $\kappa>0$ such that for each $x\in M$ the set $ I(x)$ is $\kappa$-Lipschitz continuous. Moreover for any $(x,v)\in UM$ and $(y,w)\in U_{\exp_x{(\R v)}}M$ we have 
\[   \left|t_{cut} (y,w) - t_{cut} (x,v) \right| \leq \kappa\, d_{TM}\left( (x,v), (y,w)\right).
\]
We call this property the  Lipschitz continuity in the geodesic direction.  
 %$\left\lbrace  I(x) \, | x\in M \right\rbrace$
%\item There exists  $\kappa>0$ such that for each $x\in M$ the set $ I(x)$ is $\kappa$-Lipschitz continuous.
\item If $M$ satisfies $\delta(TM)>0$ then there exists  $\kappa>0$ such that $\left\lbrace  (x,p)     \, | \,x\in M, \, p\in I(x) \right\rbrace$ is $\kappa$-Lipschitz continuous.  
\item If $M$ has dimension $2$ then there exists  $\kappa>0$ such that $\left\lbrace  (x,p)    \, |\, x\in M, \, p\in I(x) \right\rbrace$ is $\kappa$-Lipschitz continuous.
%\item In any case we have the Lipschitz continuity for any perturbation in $\left[\Ker J_0\left(t_{cut}(e_v)\right) \right]$
\end{enumerate}
\end{Thm}
To prove this theorem,  we first prove the two following results:
\begin{Thm}[Lipschitz continuity of the tangent focal loci]\label{foclip}   %ndrh ce serait bien de mettre un titre au theorem
There exists a constant $\kappa$ such that $\left\lbrace  (x,p)   \, |\, x\in M, \, p \in \NF(x)  \right\rbrace$ is $\kappa$-Lipschitz continuous.
\end{Thm}
\begin{Thm}[Semiconcavity of the tangent focal loci]\label{focsemi}~%ndrh ce serait bien de mettre un titre au theorem
The set $\left\lbrace  (x,p)    \, | \,x\in M, \, p \in \NF(x) \right\rbrace$ is semiconcave.
\end{Thm}
The definition of semiconcavity is similar as the definition \ref{Llip}, where we ask $\tau$ to be semiconcave instead of Lipschitz continuous.

\begin{Rk}
The first item of Theorem \ref{cutlip} is a result due to Li-Nirenberg, Itoh-Tanaka, and Castelpietra-Rifford \cite{it01,ln05,cr10}, while the second and third ones are new.
%They deal with the Lipschitz continuity  with respect to the manifold in some directions.
\end{Rk}

\subsection{Proof of Theorem \ref{foclip}: Lipschitz continuity of the tangent focal loci}
%%%%%We split this section in three parts; in a first part we give some basics on Jacobi fields and focalisation; in a second part we give a heuristic proof of theorem \ref{foclip}, working for example in the  case of a sphere. 
%%%%%The third part is a rigorous proof using the Hamiltonian structure hidden in the Jacobi field equation. This proof is based on the one given in the paper of Castelpietra and Rifford \cite{cr10}, the main difference is that we adopt here a Lagrangian point of view whereas Castelpietra and Rifford used an Hamiltonian point of view.
The proof uses the Hamiltonian structure hidden in the Jacobi field equation. It is based on the one given in the paper of Castelpietra and Rifford \cite{cr10}, the main difference is that we adopt here a Lagrangian point of view whereas Castelpietra and Rifford used an Hamiltonian point of view.
\subsubsection{Focalization and Jacobi fields}
Let $(x,v) \in TM $, and consider the geodesic path $\gamma_0:$  $t\in \R^+ \mapsto \exp_x\left( t v \right)$. 
We choose an orthonormal basis of $T_xM$ given by $\left(v,e_2,...,e_i,...,e_n\right)$ and define by parallel transport an orthonormal basis of $T_{\exp_x(tv)}M$:
 \[B(t)=\left(e_1(t),e_2(t),...,e_i(t),...,e_n(t)\right).\] 
We identify $T_{\exp_x(tv)}M$ with $\R^n $ thanks to the basis $B(t)$. 
By definition the Jacobi field equation along $\gamma_0$ is given by  \cite{GHL,Sakai}
\begin{align}
\label{jacobi} &\ddot J(t) + R(t)J(t)=0, \quad t\in \R^+,\\
\nonumber & J(0)=h, \quad h \in T_xM,\\
\nonumber &\dot J(0)=p, \quad p\in T_xM,
\end{align}
where $R(t)$ is the symmetric operator given, in the basis $B(t)$, by $R(t)_{ij}=\left\langle  R(e_i,e_j)e_i,e_j \right\rangle$, where $R$ is the Riemann tensor.
The Jacobi fields describe how a small perturbation of the geodesic path evolves along it. Since a focal point is related to the size of the neighborhood one can ``visit'' by perturbing the geodesic path, one can understand that both notions are linked.  
The Jacobi field equation \eqref{jacobi} is a linear equation of order two, hence we define $J_0^1: t \mapsto  M_n\left(\R \right) $ as the solution of the following matricial Jacobi field equation
\begin{align}
\nonumber &\ddot J(t) + R(t)J(t)=0, \quad t\in \R^+,\\
\nonumber & J(0)=I_{n},\\
\nonumber &\dot J(0)=0.
\end{align}
We similarly define $J_1^0$ as the solution of
\begin{align}
\nonumber &\ddot J(t) + R(t)J(t)=0, \quad t\in \R^+,\\
\nonumber & J(0)=0,\\
\nonumber &\dot J(0)=I_{n}.
\end{align}
Any solution $J$ of the Jacobi field equation \eqref{jacobi} can be written for any $t \in \R^+$
\begin{equation}\label{jacobiegalite}
J(t)= J_0^1(t) J(0) +  J_1^0(t) \dot J(0).
\end{equation}
Let us now exhibit two very particular families of Jacobi fields. For any $h\in T_xM$ we define the path
\begin{align}
\label{ja1g} \gamma_{\alpha}(s,t)&= \exp_{\exp_x(sh)}\left(tv\right), \quad &(s,t)\in [0,1]\times \R^+,\\
\label{ja2g} \gamma_{\beta}(s,t) &= \exp_x\left(t(v+sh)\right), \quad &(s,t)\in [0,1]\times \R^+.
\end{align} 

It leads to the following families of Jacobi fields
\begin{align}
\label{ja1} J_{\alpha}(t) &:= \left. \frac{d}{ds} \right|_{s=0}  \gamma_{\alpha}(s,t)= \left( d_{x}\exp_{\cdot}  (tv) \right) \cdot (h), \\
\label{ja2} J_{\beta}(t) &:= \left. \frac{d}{ds} \right|_{s=0} \gamma_{\beta}(s,t)= \left( d_{p=tv} \exp_x  \right) \cdot (th) .
\end{align} 
Notice that the Jacobi field $J_{\beta}$ is nothing but $J_1^0(\cdot) h$, since $J_{\beta}(0)=0$ and $\dot J_{\beta}(0)=h$. Analogously the Jacobi field $J_{\alpha}$ is equal to $J_0^1(\cdot) h$.
The link with focalization is enclosed in the following lemma.
\begin{Lem}\label{tflienavecjacobi}
Let $(x,v) \in U_xM$ then
\begin{equation}
t_f(x,v) = \inf \left\lbrace t \in \R^+\,   |\, \exists \,q \in U_xM \mbox{ with }  J_1^0(t)q=0.   \right\rbrace
\end{equation}
The direction $q$ is called a focal direction at $(x,v)$.
\end{Lem}
\begin{proof}
The proof is a direct consequence of \eqref{ja2}: for any $t>0$, $J_1^0(t) h=\left( d_{p=tv} \exp_x  \right) \cdot (th)$.
\end{proof}

\subsubsection{Proof of Theorem \ref{foclip} }
We start with some remarks on the symplectic structure coming with a Riemannian manifold.
\begin{Def}[The symplectic form]\label{formesymplectique}
Let M be a Riemannian manifold of dimension $n$. For any $x \in M$ we fix a base $\cal{B}$ of $\left( T_xM \times T_xM \right)$, and we define the symplectic form $\sigma$ as
\begin{align*}
\sigma: \quad \left( T_xM \times T_xM \right)^2 &\to \R,\\
(h,q)_{\cal{B}},(h',q')_{\cal{B}} & \mapsto \left\langle h, q' \right\rangle-\left\langle h', q \right\rangle = (h,q)^t \, \J \, (h',q'),
\end{align*}
where the matrix $\J=\left( 
\begin{matrix}
0& I_n \\
-I_n &0
\end{matrix}
\right)$.
A change of coordinates given by a matrix $P$ is symplectic if $P^t \J P= J$. In this case in the new base $\cal{B}'$ we  have 
\begin{align*}
\sigma: \quad \left( T_xM \times T_xM \right)^2 &\to \R,\\
(a,b)_{\cal{B'}},(a',b')_{\cal{B'}} & \mapsto \left\langle a, b' \right\rangle-\left\langle a', b \right\rangle = (a,b)^tP^t \, \J \,P (a',b'),
\end{align*}
\end{Def}
\begin{Def}[Lagrangian subspace.]\label{espaceLagrangian}
A subspace $L \in T_xM \times T_xM$ is said to be Lagrangian if $\dim(L)=n$ and $\left. \sigma \right|_{L\times L}$ is equal to $0$.
\end{Def}
For example the vertical subspace $\left\lbrace 0 \right\rbrace \times T_xM  \subset T_xM\times T_xM $ and the horizontal subspace $T_xM \times\left\lbrace 0 \right\rbrace  \subset T_xM\times T_xM $ are Lagrangian. The matrix $J_1^0 $ and $J_0^1 $ are the fundamental solutions of the Jacobi field  equation \eqref{jacobi} on those subspaces.
\begin{Lem}\label{represetationespacelagrangien}
Let $L$ be a Lagrangian subspace and $E,F$ be two vectorial spaces of dimension $n$ such that $E \stackrel{\perp}{\oplus} F=T_xM\times T_xM$ and the change of coordinates matrix is symplectic. Suppose that $L\cap E \times \left\lbrace 0 \right\rbrace = \left\lbrace 0 \right\rbrace$. Then there exist a symmetric matrix $S$ such that 
\[
L=\left\lbrace  \left( Sq ,q \right)_{E,F}\,|\, q\in F \right\rbrace.
\]
We say that $L$ is a graph above $F$.
\end{Lem}
\begin{proof}
The matrix $S$ exists since $L$ has dimension $n$ and no direction in $E$. 
To see that $S$ is symmetric we look at the symplectic form on two  vectors of $L$: let $q,q'\in F$.
Then by definition
\begin{align*}
0&=\sigma \left( \left(Sq,q \right),\left(Sq',q' \right) \right)\\
&=\left\langle Sq, q'\right\rangle-\left\langle Sq', q\right\rangle\\
&=\left\langle Sq, q'\right\rangle-\left\langle q, Sq'\right\rangle.
\end{align*}
\end{proof}
%%%%%\begin{Rk}
%%%%%That is exactly the method we used to prove lemma \ref{matricesymmetrique}.
%%%%%\end{Rk}

An important link between the symplectic form and the Jacobi field is that the symplectic form is preserved along the flow of the Jacobi  field equation.

\begin{Lem}\label{flowsymplectic}
Let $J_1$ and $J_2$ be two solution of the Jacobi field equation \eqref{jacobi}. Then for any $t>0$ 
\[
\sigma\left( \left(J_1(t)  ,\dot J_1(t)   \right),\left(J_2(t)  ,\dot J_2(t)   \right) \right)=\sigma\left( \left(J_1(0) ,\dot J_1(0)  \right),\left(J_2(0) ,\dot J_2(0)  \right) \right).
\]
Equivalently, defining $M(t) = \left( \begin{matrix}
J_1(t) &  J_2(t) \\
\dot J_1(t) & \dot J_2(t)
\end{matrix} \right)$ we have $M^t(t) \J M(t) =\J$. In this case we say that $M(t)$ is symplectic.
\end{Lem}
%%%%%\begin{proof}
%%%%%The proof is a simple computation, let \[f(t)=\sigma\left( \left(J_1(t) ,\dot J_1(t)   \right),\left(J_2(t) ,\dot J_2(t)   \right) \right).\] Since $R$ is symmetric we have  
%%%%%\begin{align*}
%%%%%\dot f (t)&= \left\langle J_1(t), \ddot J_2(t) \right\rangle-\left\langle J_2(t), \ddot J_1(t)\right\rangle\\
%%%%%&= \left\langle J_1(t), -R(t) J_2(t) \right\rangle-\left\langle J_2(t), -R(t) J_1(t)\right\rangle\\
%%%%%&=0
%%%%%\end{align*}
%%%%%\end{proof}
%%%%%\begin{Rk}
%%%%%The equality $M^t(t) \J M(t) =\J$ implies the identity \eqref{identiteconstante}, it explains why the proof of this identity is exactly the one we have just done.
%%%%%\end{Rk}
We now define a particular Lagrangian subspace in order to find a new formulation for $t_f$.
\begin{Def}\label{Lagrangianlelongduflot}
Let $(x,v)\in UM$. We define:
\begin{itemize}
\item the horizontal subspace at $\exp_x(tv)$: \[H_{t,v}:= T_{exp_x(tv)M}\times \left\lbrace 0\right\rbrace  \subset T_{exp_x(tv)M} \times T_{exp_x(tv)M}.\]
\item  the vertical subspace at $\exp_x(tv)$: \[V_{t,v}:=\left\lbrace 0\right\rbrace \times T_{exp_x(tv)M} \subset T_{exp_x(tv)M} \times T_{exp_x(tv)M}.\]
\item the subspace $L_{t,v}$ of initial conditions such that at time $t$ the Jacobi field is equal to $0$:
\[
L_{t,v}:= \left\lbrace  (h,q)\in T_xM\times T_xM\,|\, J_0^1(t)h+ J_1^0(t)q = 0  \right\rbrace.
\]
\end{itemize}
\end{Def}
Notice that $L_{t,v}$ can be equivalently defined as
$L_{t,v}= M^{-1}(t)V_{t,v}$
where  $$M(t) := \left( \begin{matrix}
J_0^1(t) &   J_1^0(t) \\
\dot J_0^1(t) & \ J_1^0(t)
\end{matrix} \right).$$

\begin{Prop}
The space $L_{t,v}$ is a Lagrangian subspace of $T_xM \times T_xM$. 
\end{Prop}
\begin{proof}
Since $M^t(t) \J M(t) =\J$ the matrix $M(t)$ is invertible, therefore $ L_{t,v}$ is a vectorial subspace of dimension $n$.

To see that it is Lagrangian we use that $\sigma$ is preserved along the flow: let $(h,q),(h',q')\in L_{t,v}$, and denote by $J_{h,q}$ the solution of the Jacobi field equation \eqref{jacobi} with $J_{h,q}(0)=h$ and $\dot J_{h,q}(0)=q$. Then, for any $t>0$,
\begin{align*}
\sigma \left( (h,q), (h',q') \right) &= \sigma \left( (J_{h,q}(t),\dot J_{h,q}(t)), (J_{h',q'}(t),\dot J_{h',q'}(t)) \right) \\
& = \sigma \left( (0,\dot J_{h,q}(t)), (0,\dot J_{h',q'}(t)) \right)=0.
\end{align*}
\end{proof}
We can now give a new formulation of Lemma \ref{tflienavecjacobi}.
\begin{Lem}\label{tflienavecjacobi3}
Let $(x,v) \in U_xM$. Then
\begin{equation}
t_f(x,v) = \inf \left\lbrace t \in \R^+ \,|\,  L_{t,v}\cap V_{0,v}\neq \left\lbrace 0 \right\rbrace    \right\rbrace.
\end{equation}
The set  $L_{t,v}\cap V_{0,v}$ is called the focal set at $(x,v)$.
\end{Lem}
\begin{proof}
Let $q \in U_xM\setminus\{ 0\}$ satisfy $(0,q) \in L_{t,v}\cap V_{0,v}$. Then $J_{0,q}(t)=J_1^0(t)q=0$ and Lemma \ref{tflienavecjacobi} concludes the proof.
\end{proof}

%%%%%\begin{Rk}
%%%%%In the heuristic proof the hypothesis "$J_0^1(t_f(x,v))$ invertible" exactly says that $L_{t_f(x,v),v}$ is a graph above $V_{0,v}$, given by the application $$S(t_f(x,v))=\left( J_0^1(t_f(x,v))\right)^{-1} J_1^0(t_f(x,v)).$$ In the general case $J_0^1(t_f(x,v))$ has no reason to be invertible, to adapt the proof we need to find another way to write $L_{t_f(x,v),v}$ as a graph.
%%%%%\end{Rk}
We recall that we identify $T_{\exp_x(tv)}M$ with $\R^n$ through the basis \[B(t)= \left( e_1(t),...,e_i(t),...,e_n(t) \right).\]
According to Lemma \ref{represetationespacelagrangien} the obstruction to see $L_{t_f(x,v),v}$ as a graph above $V_{0,v}$ comes from the intersection of $L_{t_f(x,v),v}$ with the horizontal space. By definition we have 
\[
L_{ t_f (x,v),v}\cap H_{0,v}=\Ker\, J_0^1(t_f(x,v)).
\]
Let us identify, for any $u \geq 0$, $H_{u,v}$ with $\vect \left( e_1'(u),...,e_i'(u),...,e_n'(u) \right)$
and $V_{u,v}$ with
 $\vect \left( f_1(u),...,f_i(u),...,f_n(u) \right),$ where $e'_i(u)=e_i(u) \times \left\lbrace 0 \right\rbrace \in T_{\exp_x(uv)}\times T_{\exp_x(uv)}$
and $f_i(u)=  \left\lbrace 0 \right\rbrace \times e_i(u)  \in T_{\exp_x(uv)}\times T_{\exp_x(uv)}$.
With this notation, without loss of generality we can suppose there exists an index $l>1$ such that $\Ker \,J_0^1(t_f(x,v))=\vect \left( e'_l,...e'_n \right)$.
Therefore, for any $i\geq l$ we can change $e_i'(u)$ by $f_i(u)$ and $f_i(u)$ by $-e_i'(u)$ to get two new orthonormal spaces of dimension $n$:
\begin{align*}
E(u)=\vect \left( e_1'(u),...,e_{l-1}'(u),f_l(u),...,f_n(u) \right)\\
F(u)=\vect \left( f_1(u),...,f_{l-1}(u),-e_l'(u),...,-e_n'(u) \right).
\end{align*}
\begin{Rk}
The change of coordinates is symplectic, that is $P^t \J P= J$, where $P$ is the change of basis matrix. Therefore for any $(z,w),(z',w')\in E \times F$ we have 
\[\sigma\left( (z,w),(z',w') \right)=\left\langle z, w'\right\rangle-\left\langle z', w\right\rangle.\]
\end{Rk}
By construction for any $u\geq 0$ we have 
\begin{enumerate}
\item $E(u) \stackrel{\perp}{\oplus}F(u) = T_{\exp_x(uv)}\times T_{\exp_x(uv)}$
\item $L_{ t_f (x,v),v}\cap E(0)=\left\lbrace 0 \right\rbrace$.
\end{enumerate}
Since $L_{ u,v'}$ is smooth with respect to $(x',v',u)$, there exist a neighbourhood $O_{x,v,t_f(x,v)}\subset TM\times \R^+$  of $\left(x,v,t_f(x,v)\right)$ such that, for any $\left(x',v',t\right)\in  O_{x,v,t_f(x,v)}$, we have
\begin{equation}\label{orthogonal}
L_{t,v'}\cap E(0)=\left\lbrace 0 \right\rbrace.
\end{equation} 
Moreover Lemma \ref{represetationespacelagrangien} implies that there exist a smooth function
\begin{align*}
S:\quad O_{x,v,t_f(x,v)} &\to S_n\left( \R \right)\\
   \left(x',v',t\right)  &\mapsto S(t),
\end{align*}
such that, for any $w\in F(0)$, we have $S(t)w \in E(0)$ and 
\begin{equation*}
L_{t,v'}= \left\lbrace \left( S(t)w , w\right)_{E(0) \times F(0)} \mbox{ with } w\in F(0)  \right\rbrace.
\end{equation*}
\begin{Rk}
Notice that the matrix $S(t)$ as well as the subspaces $E(u)$ and $F(u)$
depend on $(x',v',t)$, but the indices $l$ used to define  $E(u)$ and $F(u)$ for any $(x',v',t) \in O_{x,v,t_f(x,v)}$ only depend on $x,v,t_f(x,v)$.
\end{Rk}
%\begin{Rk}\label{formesymplecidem}
%The matrix $S(t)$ depends on $(x',v',t)$, The subspaces $E(u)$ and $F(u)$ also depend on $(x',v',t)$, but the indices $l$ used to define  $E(u)$ and $F(u)$ for any $(x',v',t) \in O_{x,v,t_f(x,v)}$ only depends on $x,v,t_f(x,v)$.
%\end{Rk}
The following lemma is the key tool to apply later the Implicit Function Theorem.
\begin{Lem}\label{tfjacobi4}
let $(x,v)\in TM$. 
\begin{enumerate}
\item Let $q\in U_xM$ satisfy $(\left\lbrace 0 \right\rbrace,q)\in L_{ t_f (x,v),v}\cap V_{0,v} $. Then $q \in F(0)$ and $q^t\, S(t_f (x,v)) q=0$.
\item There exists $\delta >0$ such that, for any $(x,v)\in TM$,  $|| \dot S(t_f(v)) ||\geq \delta$.
\item For any $(x',v',t) \in O_{x,v,t_f(x,v)}$, if $q^t\, S(t) q=0$ then $t_f(v')\leq t$.

\end{enumerate}
\end{Lem}
Notice that $q$ is defined only in $T_xM$, but for any $x'$ close to $x$ we can see it also as an element of $T_{x'}M$ thanks to the identification with the coordinates. The dot always stands for the derivative along the Jacobi Field ($\frac{d}{dt} $).

\begin{proof}
Let $q\in L_{ t_f (x,v),v}\cap V_{0,v} $. Since $L_{ t_f (x,v),v}\cap H_{0,v}=\vect \left( e'_l,...e'_n \right)$, using the symplectic form $\sigma$ we find that $q_i=0$ for any $i=l,\ldots,n$. This gives that $q \in F(0)$. Moreover $S(t_f (x,v)) q \in V_{0,v} $ thus 
$\left( S(t_f (x,v)) q \right)_i=0$ for any $i=1,\ldots,l-1$. Consequently $q^t\, S(t_f (x,v)) q=0$.

To compute the derivative with respect to $t$ we again use the symplectic form. Let $(0,z)\in V_{t,v}$ for any $t$ such that $(x,v,t) \in O_{x,v,t_f(x,v)}$. There exists $\phi(t)=(h_t,q_t)=(S(t)w_t,w_t)_{E(0)\times F(0)} \in L_{t,v}$  such that $M(t)\phi(t)= (0,z)$. On one side
\begin{align*}
\sigma\left( \phi(t),\dot \phi(t) \right)&=\sigma\left( (S(t)w_t,w_t)_{E(0)\times F(0)},(\dot S(t)w_t +  S(t) \dot w_t,\dot w_t)_{E(0)\times F(0)} \right) \\
&=\sigma\left( (S(t)w_t,w_t),( S(t) \dot w_t,\dot w_t) \right)+\sigma\left( (S(t)w_t,w_t),(\dot S(t)w_t ,0) \right)\\
&=\left\langle \dot S(t)w_t,w_t\right\rangle.
\end{align*}
On the other side,
\begin{align*}
\sigma\left( \phi(t),\dot \phi(t) \right)&= \sigma\left( M(t) \phi(t),M(t) \dot \phi(t) \right).
\end{align*}
Since $ M(t)\phi(t)= (0,z) $ we have $M(t) \dot \phi(t) =-\dot M(t)\phi(t) $. Moreover 
\[\dot M(t)=\left( \begin{matrix}
0 & I_n \\
-R(t) & 0
\end{matrix} \right) ,\]
thus
\begin{align*}
\sigma\left( \phi(t),\dot \phi(t) \right)&= -\sigma\left( M(t) \phi(t),\dot M(t)\phi(t) \right) \\
&=-\sigma\left( M(t) \phi(t),\left( \begin{matrix}
0 & I_n \\
-R(t) & 0
\end{matrix} \right) M(t)\phi(t) \right) \\
&=-\sigma\left( (0,z),(-z,0) \right)=-|z|^2.
\end{align*}
Hence
\[
\left\langle \dot S(t)w_t,w_t\right\rangle=-|z|^2,
\]
and by compactness we deduce the existence of a constant $\delta>0$ such that $|| \dot S(t_f(x,v)||\geq \delta$.

For the third item we reason by contradiction: we take $(x',v',t') \in O_{x,v,t_f(x,v)}$ and suppose that $q^t\, S(t') q=0$ and $t'< t_f(x',v')$. By definition $q \in V_{0,v'}\cap F(0)$, thus $ q_i=0$ for any $i=l,\ldots,n$. Since $t'< t_f(x',v')$ the space $L_{ t' (x,v),v}$ is a graph on the horizontal space. More precisely, according to \eqref{jacobiegalite}, for any $t\in (0,t_f(x',v'))$ we have
\[
L_{ t_f (x,v),v}=\left\lbrace ( h, \left( J_1^0(t_f(x,v))\right)^{-1} J_0^1(t) h)\, |\, h\in H_{0,v} \right\rbrace.
\]
We denote $\left( J_1^0(t_f(x,v))\right)^{-1} J_0^1(t)=K(t)$. Then the exact same computation done above proves that
\[
\left\langle \dot K(t)h,h\right\rangle<0\qquad \forall\,h\in H_{0,v}.
\]
Since $t\left( J_1^0(t_f(x,v))\right)^{-1}$ converges to $I_n$ when $t$ goes to zero, we deduce that for $t$ small enough $K$ is symmetric positive definite.\\
For any $h \in H_{0,v}$ and $q'\in V_{0,v}$  we denote $h=(h_1,h_2)$, where $h_1\in  H_{0,v}\cap E(0)$, $h_2\in  H_{0,v}\cap F(0)$, and $q'=(q_1',q_2')$ with $q_1'\in  V_{0,v}\cap E(0)$ and $q_2'\in  V_{0,v}\cap F(0)$. With this notation we have $\left( (h,q')=(h_1,q_2'),(q_1',h_2)\right)_{E\times F}$ and we define
the matrices $S_i(t)$, $K_i(t)$, $i=1,\ldots,4$, such that
\[
(q_1',q_2')= \left( K_1(t)h_1+K_2(t)h_2,K_3(t)h_1+K_4(t)h_2 \right),
\]
and 
\[
(h_1,q_2')= \left( S_1(t)q_1'+S_2(t)h_2,S_3(t)q_1'+S_4(t)h_2 \right).
\]
Since by hypothesis $L_{t',v'}$ is a graph over $H_{0,v}$ and $F(0)$, we deduce that $S_1(t')=K^{-1}_1(t')$ and in particular we see that $K_1(t')$ is invertible. In the focal direction $q\in F(0)\cap V_{0,v'}$ we have $q=(q_1,0)_{F(0)}$ and
\[
0=q^t\, S(t') \, q= q_1^t\, S_1(t') \, q_1 = h_1^t\, K_1(t') \, h_1,
\]
where $h_1(t)=K^{-1}_1(t)q_1$. To get a contradiction we just have to remark that, for any $A>0$, up to taking $O_{x,v,t_f(x,v)}$ smaller we have that, for any $(x',v',t) \in O_{x,v,t_f(x,v)}$ with $t \leq t_f(x',v')$, \[h_1^t (t)K_1(t) h_1(t) \leq -A h_1^t(t) h_1(t).\] Also, in the direction $(x,v)$, for any $t\leq t_f(x,v)$ we have
\[\bigl( \left(S_1\left(t\right)q_1,S_3\left(t\right)q_1\right),\left(q_1,0\right)\bigr)_{E\times F}=\bigl(\left(h_1\left(t\right),0\right),\left(K_1\left(t\right)h_1\left(t\right),K_3\left(t\right)h_1\left(t\right)\right) \bigr) \in L_{t,v}.\] 
By definition of $q$ we have $S_1(t)q_1=h_1(t) \to 0$ when $t\to t_f(x,v)$ and $K_1(t)h_1(t)=q_1 $. Assuming with no loss of generality that  $K_1(t)$ is diagonal, we see that any eigenvalue $\lambda_i(t)$ corresponding to an eigenvector $q_i\neq 0$ goes to $-\infty$ (notice that it cannot goes to $+\infty$ since we proved that $t\mapsto K(t)$ decreases). Hence, being the eigenvalues continuous with respect to $(x',v',t)$, by further shrinking $O_{x,v,t_f(x,v)}$ if needed we have $h_1^t(t) K_1(t) h_1(t) \leq -A h_1^t h_1$.
\end{proof}
\begin{Rk}
The last proof just says that, before focalization,  when the Lagrangian space $L_{t,v}$ has a vertical component it cannot be at the same time a graph above the horizontal space and above $F$.
\end{Rk}
To conclude the proof of Theorem \ref{foclip} we apply the Implicit Function Theorem in order to find the function $\tau$ needed in Definition \ref{Llip}. 
%We define $\Psi$ by:
Let $(\bar x,\bar v)\in UM$ and $q\in U_xM$ be the focal direction associated. Then the function 
\begin{align*}
\Psi:\quad UM\times \R^+ &\to \R \\
(x,v,t) &\mapsto q^t S\left(t \right)q
\end{align*}
is well defined on a neighbourhood of $(\bar x,\bar v,t_f(\bar x,\bar v))$. Moreover $\Psi(\bar x, \bar v,t_f(\bar x,\bar v))=0$ and by Lemma  \ref{tfjacobi4} we have:
\[\left| \partial_t \Psi(\bar x, \bar v,t_f(\bar x,\bar v)) \right|=\left| q^t \dot S  \left(t_f(x,v)\right) q \right|\geq \delta.\]
Hence, by the Implicit Function Theorem we get a function $\tau$ defined in a neighborhood $O_{\bar x,\bar v}$ of $(\bar x,\bar v)$ such that $\Psi( x,  v,\tau( x, v))=0$. By Lemma \ref{tfjacobi4} we find that $t_f(x,v)\leq \tau(x,v)$, and
it only remains to check that $\tau$ is Lipschitz continuous. 
This follows from the fact that, by compactness, there exist $K>0$ such that
\begin{align*}
\left| d_{\bar x,\bar v}\tau \right| &=\left| \frac{1}{\partial_t \Psi(\bar x, \bar v,t_f(\bar x,\bar v))}d_{\bar x,\bar v}\Psi \right|\\
&\leq \frac{K}{\delta}.
\end{align*}
It concludes proof of Theorem \ref{cutlip}.
\begin{Rk}
This method also proves Theorem \ref{focsemi}, as we easily see that the second differential of $\tau$ at $(x,v)$ is bounded from above.
\end{Rk}
\subsection{Proof of Theorem \ref{cutlip}: Lipschitz continuity of the tangent cut loci}
\begin{proof}[Proof of Theorem \ref{cutlip}]
Let $x \in M$, $e_v\in U_xM$, $v=t_{cut}(e_v)e_v$. 
We want to find the function $\tau$ needed in Theorem \ref{Llip} using the Implicit Function Theorem. The construction of the function $\tau$ will depends on $x,v$, and $\delta(v)$.

\subsubsection{At the intersection with the tangent focal locus}

If $v \in \TFL(x) \cap \TCL(x)$ then $t_{cut}(x,e_v)=t_f(x,e_v)$ and for any $\left(y,e_w \right) \in U_xM$ we have $t_{cut}(y,e_w)\leq t_f(y,e_w)$. 
Notice that by Theorem \ref{foclip} the function $t_f$ is $\kappa$--Lipschitz continuous, so the choice $\tau =t_f$ works.

\smallskip

%This shows that in the non focal case, the function $t_{cut}$ is Lipschitz continuous on $UM$.
 
\subsubsection{Far from the tangent focal locus}  %ndrh Points far?

If $v \not \in \TFL(x) \cap \TCL(x)$ then  $\delta(v) > 0$.  
 Let $\bar v \in \bar \I(x)$ such that $|v- \bar v|=\delta(v)$ and $\exp_x v=\exp_x \bar v = y$.
Let $K \subset TM$ be a compact neighborhood of the geodesic path $t\in [0,1] \mapsto \exp_x(t \bar v)$ and $0<\epsilon<t_{inj}$  such that $ B(y,\epsilon)\subset K(y)$. For any $\eta \in T_yS$ with $z=\exp_y \eta \in B(y,\epsilon)$, we construct a path $s,t \in [0,\epsilon]\times [0,1] \mapsto \gamma(s,t)$ satisfying the following conditions for any $(s,t)\in [0,\epsilon]\times[0,1]$:
\begin{enumerate}
\item $\gamma(0,t)=\gamma(t)=\exp_x(t\bar v)$.
\item $\gamma(s,1)= \exp_y(s \eta)=z_s$.
\item $\gamma(s,0)= x$.
\item $\gamma(\cdot,\cdot) \in C^1([0,1]^2,M)$.
\item $(\gamma(s,t),\dot \gamma(s,t)) \in K$.
\end{enumerate}
Working in smooth charts this construction is easy to realize. %ndrh comprends pas smooth charts
Note that $s\leq \epsilon \leq t_{inj}$ implies that $s \mapsto \exp_y(s \eta)$ is a minimizing geodesic path, therefore $d^2(y,z_s)=s^2$ and $z_s \in B(y,\epsilon)$. However  $t \mapsto \gamma(s,t)$ and $s \mapsto \gamma(s,t)$ are not necessarily geodesic paths away from $s=0$ and $t=1$. Anyway the first variation formula %\ref{ffv2} of \ref{ffvfondamental}
 applied to $\gamma$ provides a constant $K$ such that
%As $t_{cut}(e_{\bar v}) < t_f(e_{\bar v})$, $d_{p=\bar v}\exp_x$ is invertible so taking $\epsilon $ smaller if we need,
%we can map all $B(y,\epsilon)$ starting from a neighberhood of $(x,\bar v)$, we note $\bar \eta \in B(\bar v,\bar \epsilon)$ the vector such that $\exp_x(\bar v+\bar \eta)=z$. Let $\gamma_{\bar v}(u,z)=\exp_x(u(\bar v+\bar \eta)) $ the geodesic path starting at $x$ at time $0$ and reaching $z$ at time $1$. By the formula of first variation we get for $z$ small enough :
\begin{align}\label{var1}
d^2(x,z_s)&\le A(\gamma(s,t)) \le A(\gamma(0,t))+ s\left\langle d_{\bar v}\exp_x \bar v,\eta \right\rangle +\frac{s^2}{2}K \\
\nonumber &\leq d^2(x,y) + s\left\langle d_{\bar v}\exp_x \bar v,\eta \right\rangle + \frac{K}{2}d^2(y,z_s).
\end{align}
We can similarly add a perturbation of $x$. 
Hence we define $u : B(x,\epsilon) \times B(y,\epsilon) \to \R^+$ by 
\begin{multline}\label{majorationdeu}
u(x',z):= d^2(x,y)+ \langle d_{\bar v}\exp_x \bar v,(\exp_y)^{-1}(z) \rangle -\langle  \bar v,(\exp_x)^{-1}(x') \rangle \\+ K \left( d^2(x,x')+ d^2(y,z) \right).
\end{multline}
Note that if we compare the above expression to the right hand side of \eqref{var1} we have changed $\frac12 K$ to $K$; this modification shows that $d^2(x',z)=u(x',z)$ if and only if $z=y$ and $x'=x$ otherwise $d^2(x',z)<u(x',z)$. 
Moreover $u$ is $C^1$ and 
\[(d_{x'=x,z=y}u)\cdot(\zeta, \eta) =-\langle  \bar v,\zeta \rangle + \langle d_{\bar v}\exp_x \bar v,\eta \rangle.\]
By continuity of  $\exp_x$ there exits $\epsilon>0$ such that for any $(x',w)\in B \left( (x,v),\epsilon \right)\subset TM $, $\exp_{x'}(w)=z \in B(y,\epsilon) $. 
Let $\gamma(x',w,\theta)=\exp_{x'}(\theta w) $ we define $\Phi$ : $B\left( (x,v),\epsilon \right) \to \R$ by
\[
w \mapsto u(\exp_{x'} (w))-A(\gamma(x',w,\theta)).
\] 
According to the first variation formula %\eqref{ffv1} of \eqref{ffvfondamental}
, $\Phi$ is $C^1$ on $B \left( (x,v),\epsilon \right) $ and the differential at $x,v$ in the direction $\zeta,\xi$ ({\it i.e. } $x'=\exp_x(r\zeta)$, $w=v+s\xi$) is given by
\begin{align}\label{differentielenw}
(d_{x,v}\Phi)(\zeta,\xi)&= \langle d_{p=\bar v}\exp_x \bar v,d_{p=v}\exp_x \xi\rangle_y- \langle d_{p=v}\exp_x v,d_{p=v}\exp_x \xi \rangle_y+\langle v- \bar v,\zeta \rangle\\
\nonumber &=\langle  q-\bar q,\eta\rangle_y  +\langle v- \bar v,\zeta \rangle,
\end{align}
where $d_{\bar v}\exp_x \bar v=-\bar q$, $d_{v}\exp_x v=-q$ and $d_v\exp_x \xi =\eta$.
%Indeed $\gamma_i(s,u)=\gamma(v+s\xi,u)$ satisfies $\gamma_i(s,0)=x$ and is geodesic for $s=0$. \\

The set $O_{x,v} :=\left\lbrace  (x',v',t)\in  UM \times \R^+ \,|\, (x',tv' ) \in B \left( (x,v),\epsilon \right)\right\rbrace $  is an open subset of $ UM \times \R^+ $, and $(x,e_v,t_{cut}(e_v)) \in O_{x,v}$. 
We define  $\Psi$ by
\begin{align*}
\Psi: \quad O_{x,v} &\to \R \\
\Psi(x',v',t) & \mapsto \Phi(x',tv').
\end{align*}
By definition $\Psi(x,e_v,t_{cut}(e_v))=u(x,y)-A(\gamma(x,v,\theta))=0$ and for $(x',v',t)\neq (x,e_v,t_{cut}(e_v)) $ if  $\Psi(x',v',t)=0$ then \eqref{var1} implies \[d^2(x',\exp_{x'}(tv'))<A(\gamma(x',v',t)),\] hence $t > t_{cut}(x',v')$. 
Furthermore we compute 
\[\frac{\partial}{\partial t} \Psi (x,e_v,t_{cut}(e_v))=    d_{p=v}\Phi(x,e_v)=\langle q-\bar q,-\frac{1}{t_{cut}(e_v)}q\rangle_y.\] 
Since the geodesic flow is Lipschitz continuous, there exists $A>0$ such that
\[\frac{1}{A}\leq |q-\bar q|_y \leq A|v-\bar v|_x. \]
Since  $|q|^2_y=|\bar q|^2_y$,  and $t_{cut}$ is bounded by a constant $C$ on $TM$,  we have
\begin{equation}\label{estdt2}
\frac{1}{t_{cut}(e_v)}|\langle q-\bar q, q\rangle_y |=\frac{1}{t_{cut}(e_v)}| q-\bar q|^2 \geq \frac{1}{2AC} \delta(v)^2> 0.
\end{equation}
Therefore 
\begin{equation}\label{estdt}
\left| \frac{\partial}{\partial t} \Psi (x,e_v,t_{cut}(e_v)) \right| \geq \frac{1}{2C'} \delta(v)^2> 0.
\end{equation}
Consequently we can apply the implicit function theorem to $\Psi(x',v',t)=0$ at $(x,e_v,t_{cut}(e_v))$, to find a neighborhood of $O'_{x,v} \subset UM$ of $(x,e_v)$ and a function $\tau \in C^1(O'_{x,v},\R^+)$ such that 
\begin{eqnarray}
\forall (x',v') \in O'_{x,v}\quad \, t_{cut}(x',v') \leq \tau(x',v'), \qquad  t_{cut}(x,e_v)=\tau(x,e_v).
\end{eqnarray}
The implicit function theorem also gives the differential of $\tau$:
\begin{align}\label{diff}
d_{x'=x,v'=e_v}\tau (\zeta,\xi)&=-\frac{1}{d_{p=v}\Phi(x,e_v)} d_{x'=x,p=v}\Phi(\zeta,\xi)\\
\nonumber &=\frac{{t_{cut}(e_v)}}{\langle q-\bar q,q\rangle_y}\left[  \langle q-\bar q,\eta \rangle_y + \langle v-\bar v,\zeta \rangle_x \right]\\
\nonumber &\leq \frac{C''\left(|\eta|_y + |\zeta|_x\right)}{\delta(v)}.
\end{align}
%Where we used \ref{estdt} and $\langle q-\bar q,\eta \rangle_y \leq \delta(v) $ for the last inequality.\\
We fix a small constant $\bar \delta> 0$ and  distinguish two cases.
\subsubsection*{Case $1$: $\delta(v) \geq \bar \delta$} 
In this case \eqref{diff} becomes:
\[
|d_{x'=x,v'=e_v}\tau (\zeta,\xi)| \leq \frac{C}{\bar\delta}\left(|\zeta|+|\xi|\right).
\]
Therefore the function $\tau$ is $\kappa$ Lipschitz--continuous, near $(x,e_v)$, for any $\kappa \leq \frac{C}{2\bar \delta}$. 
In this case we are done. We remark that we proved the Lipschitz continuity of $t_{cut}$ for any perturbation of $(x,v)$, so in particular we obtained also the second item of Theorem \ref{cutlip} in the case $\delta(v)\geq \bar\delta$.
So we are only left to understand the case of   speeds near a purely focal point. 

\subsubsection*{Case $2$: $\delta(v) \leq \bar\delta$}
In this case $v$ is near a purely focal point, and
we need to be slightly more precise regarding the estimate of $|d_{x'=x,v'=e_v}\tau (\zeta,\xi)|$.
First of all we can rewrite \eqref{estdt2} as
\begin{equation}\label{estdt3}
\left| \frac{\partial}{\partial t} \Psi (x,e_v,t_{cut}(e_v)) \right |\geq \frac{1}{2C'}|v-\bar v|^2.
\end{equation}
Since the symplectic form is preserved along the Jacobi field we have,
for any $t>0$,
\begin{multline}\label{estdt4}
\sigma \left( (0,v-\bar v),(\zeta,\xi) \right) \\ = \sigma \left( (J_1^0(t)(v-\bar v), \dot J_1^0(t)(v-\bar v) ),( J_0^1(t)\zeta+ J_1^0(t)\xi,\dot J_0^1(t)\zeta+\dot J_1^0(t)\xi) \right) 
\end{multline} 
thus
\begin{multline}\label{equationproblem}
 -\left\langle  v-\bar v,\zeta \right\rangle_x -  \left\langle J_0^1(t)\zeta+ J_1^0(t)\xi,\dot J_1^0(t)(v-\bar v) \right\rangle_y = \\ \left\langle J_1^0(t)(v-\bar v),\dot J_0^1(t)\zeta+\dot J_1^0(t)\xi \right\rangle_y.
\end{multline} 
%\[
%d_{p=e_v} \Psi(t_{cut}(e_v),\cdot) \xi=  d_{p=v}\Phi(\nu)=\langle q-\bar q,\eta\rangle_y,
%\]
%and note that $\xi$ is a tangent vector of $U_xM$ at $e_v$ so $\langle q,\eta\rangle_y =t_{cut}(e_v)\langle e_v,\xi\rangle_x = 0$.
A Taylor formula together with the fact that $\exp_x(v)=\exp_x(\bar v) $ gives that there exists $A\in \R_+$ such that
\begin{equation}\label{equationproblem2}
|d_{p=v}\exp_x \left( v-\bar v \right)|_y = |J_1^0(t_{cut}(e_v))(v-\bar v)|_y \leq A | v-\bar v |^2.
\end{equation}
Thus the right hand side of \eqref{equationproblem} is smaller then $A | v-\bar v |^2 $. 
Thanks to \eqref{diff},   we can show the Lipschitz continuity separately on each variable;
 we conclude by examining three different cases. 
The first case is a perturbation along the variable $v$.
The second and third cases deal with a perturbation along the variable $x$.
\\

$\bullet$ If we only consider a perturbation along the speed $(\zeta =0)$ then \eqref{equationproblem} and \eqref{equationproblem2}  give 
\begin{equation}\label{equationproblem3}
 \left| \left\langle \eta,\dot J_1^0(t)(v-\bar v) \right\rangle_y   \right| \leq A | v-\bar v |^2.
\end{equation}
Moreover a Taylor formula on $ q-\bar q= d_{p=v}\exp_x(v)- d_{p=\bar v}\exp_x(\bar v)$ gives, for $\delta(v)$ small enough,
\begin{equation}\label{besoin}
\dot J_1^0(t)(v-\bar v)=q-\bar q +o(|v-\bar v|^2).
\end{equation}
We deduce that there exist $C>0$ and $\bar\delta>0$ such that for any $x\in M$ and $v\in \I(x)$ with $\delta(v)\leq \bar\delta$ we have, according to \eqref{differentielenw},
\[
 \left|d_{x'=x,p=v}\Phi(0,\xi) \right| = \left | \left\langle \eta,q-\bar q \right\rangle_y \right| \leq C | v-\bar v |^2.
\]
Together with \eqref{estdt3} and \eqref{diff}, we obtain
\begin{equation*}
d_{x'=x,v'=e_v}\tau (0,\xi) \leq \frac{ 2C' C | v-\bar v |^2}{| v-\bar v |^2}\leq C,
\end{equation*}
which proves the Lipschitz continuity in the $v$ variable. 
We recall that the constant $C$ can change in each inequality but is uniform on $TM$.  
\\

We now want to look for the Lipschitz continuity in the $x$ variable.
\smallskip

$\bullet$ If the perturbation $\zeta $ is collinear to $v$ ($\zeta=\pm v$) then \eqref{differentielenw} rewrites
\[
|d_{x'=x,p=v}\Phi(\zeta,0)| = | \left\langle v,v-\bar v \right\rangle_x | =\frac12 | v-\bar v |_x^2.
\]
Together with \eqref{estdt3} in \eqref{diff} we obtain that 
\begin{equation*}
d_{x'=x,v'=e_v}\tau (\zeta,0) \leq  C.
\end{equation*}
This is exactly the Lipschitz continuity at $(x,v)$ in the $x$ variable along the geodesic direction given by $v$,
and this concludes the proof of the first item of Theorem \ref{cutlip}.\\

$\bullet$ If the perturbation $\zeta $ belongs to $\Ker\, J_0^1\left( t_{cut}(e_v) \right) $ and $\xi=0$, then \eqref{equationproblem} becomes
\begin{equation}\label{equationproblem4}
 -\left\langle  v-\bar v,\zeta \right\rangle_x  = \left\langle J_1^0(t_{cut}(e_v))(v-\bar v),\dot J_0^1\left( t_{cut}(e_v) \right)\zeta \right\rangle_y,
\end{equation} 
and together with \eqref{differentielenw} and \eqref{equationproblem2} we get $|d_{x'=x,p=v}\Phi(\zeta,0)| \leq A| v-\bar v |_x^2$. By this estimate combined to  \eqref{estdt3} in \eqref{diff} we obtain a constant $C>0$ such that
\begin{equation*}
d_{x'=x,v'=e_v}\tau (\zeta,0) \leq C.
\end{equation*}
Therefore the function $t_{cut}$ is Lipschitz continuous along these directions.\\

Notice that in dimension two, for any $(x,v)\in M$ we can take a basis with one direction along $e_v$ and the other one in  $\Ker \,J_0^1\left( t_{cut}(e_v) \right) $, and we  deduce that $t_{cut}$ is Lipschitz continuous on $UM$. 
This concludes the proof of Theorem \ref{cutlip}.
\end{proof}

%{\bf NdLR: stp peux-dire dire pour chaque cas ci-dessus quel item tu prouves du théorèmeB.2. Au debut c'est le (1) et ensuite la fin du (1). Le (2) était démontre avant. Et le (3) en dernier. LA fin de la preuve du (3) n'est pas clair. Dis exactement d'ou vient la formule après B.23 et aussi comment intervient B.22. }

%\begin{Rk}\label{september29}
%The Lipschitz continuity in the $x$ variable in the geodesic direction  has its own importance, in particular it allows us to show Lemma \ref{lem2}.
%\begin{Lem*}
%Let $(M,g)$ be a compact Riemannian manifold. There exist $K >0$ such that,
%for every $(x,v) \in TM$,
%\[
%K^{-1} \rho_x \bigl(v,\I(x)\bigr) \leq   \rho_y \bigl(w,\I(y)\bigr) \leq K   \rho_x \bigl(v,\I(x)\bigr),
%\]
%where $y=\exp_x(v)$ and $w = -d_v\exp_x(v)$. 
%\end{Lem*}
%\end{Rk}
\begin{Rk}
We do not know if in any dimension the function $t_{cut}$ is Lipschitz continuous on $UM$. 
However, for any $n$--dimensional Riemannian manifold, such that 
\[\dim \left[\Ker J_0^1\left(t_{cut}(e_v)\right) \right]=n-1, \] we proved that $t_{cut}$ is Lipschitz continuous on $UM$. 
 It is for example the case of $\mathbb S^n$.
 More generally we proved the following theorem:
 \begin{Thm}[Lipschitz continuity of the tangent cut loci II]
 There exists  $\kappa>0$ such that for each $x\in M$ the set $ I(x)$ is $\kappa$-Lipschitz continuous. Moreover, for any $(x,v)\in UM$, $\zeta \in \{ \Ker J_0^1\left(t_{cut}(e_v)\right)  \}\cup  \{\pm v\}$,  and $(y,w)\in U_{\exp_x{(\R \zeta)}}M$, we have 
\[   \left| t_{cut} (y,w) - t_{cut} (x,v) \right| \leq \kappa\, d_{TM}\left( (x,v), (y,w)\right). \]

  %Lipschitz continuity for any perturbation in $\left[\Ker J_0^1\left(t_{cut}(e_v)\right) \right]$.
 \end{Thm}

\end{Rk}
\end{appendices}

\end{document}